\documentclass[11pt,reqno]{amsart}
\usepackage{amsfonts, amssymb, amsmath, enumerate}
\usepackage{fullpage}
\usepackage{color}
\usepackage[linktoc=page, colorlinks, linkcolor=blue, citecolor=blue]{hyperref}

\numberwithin{equation}{section}

\DeclareMathOperator{\E}{\mathbb{E}}

\DeclareMathOperator*{\Span}{span}
\DeclareMathOperator*{\im}{Im}

\DeclareMathOperator{\dist}{dist}

\DeclareMathOperator{\supp}{supp}
\DeclareMathOperator{\sm}{sm}
\DeclareMathOperator{\Loc}{Loc}

\renewcommand{\Pr}[2][]{\mathbb{P}_{#1} \left\{ #2 \rule{0mm}{3mm}\right\}}
\newcommand{\ip}[2]{\left\langle#1,#2\right\rangle}
\newcommand{\real}[1]{\widetilde{#1}}
\newcommand{\vectwo}[2]{\binom{#1}{#2}}
\newcommand{\mattwo}[2]{\begin{bmatrix} \phantom{x} #1 \phantom{x} \\ \phantom{x} #2 \phantom{x} \end{bmatrix}}
\newcommand{\smallmattwo}[2]{\bigl[ \begin{smallmatrix} \phantom{x} #1 \phantom{x} \\ \phantom{x} #2 \phantom{x} \end{smallmatrix} \bigr]}

\def \C {\mathbb{C}}
\def \N {\mathbb{N}}
\def \P {\mathbb{P}}
\def \R {\mathbb{R}}

\def \Z {\mathbb{Z}}

\def \BB {\mathcal{B}}
\def \CC {\mathcal{C}}
\def \DD {\mathcal{D}}
\def \EE {\mathcal{E}}

\def \LL {\mathcal{L}}
\def \MM {\mathcal{M}}
\def \NN {\mathcal{N}}

\def \a {\alpha}

\def \g {\gamma}
\def \e {\varepsilon}
\def \d {\delta}
\def \l {\lambda}
\def \s {\sigma}

\def \tran {\mathsf{T}}

\def \< {\langle}
\def \> {\rangle}
\def \one {{\bf 1}}

\def \smin {s_{\min}}
\def \Comp {{\mathrm{Comp}}}
\def \Incomp {{\mathrm{Incomp}}}

\newtheorem{theorem}{Theorem}[section]
\newtheorem{proposition}[theorem]{Proposition}
\newtheorem{corollary}[theorem]{Corollary}
\newtheorem{lemma}[theorem]{Lemma}

\newtheorem{definition}[theorem]{Definition}

\theoremstyle{remark}
\newtheorem{remark}[theorem]{Remark}

\newtheorem{assumption}[theorem]{Assumption}

\title[]{No-gaps delocalization for general random matrices}

\author{Mark Rudelson \and Roman Vershynin}

\date{\today}

\address{Department of Mathematics, University of Michigan, 530 Church St., Ann Arbor, MI 48109, U.S.A.}
\email{\{rudelson, romanv\}@umich.edu}
\thanks{Partially supported by NSF grants DMS 1161372,  1265782, 1464514, and USAF Grant FA9550-14-1-0009.}

\begin{document}

\maketitle

\begin{abstract}
 We prove that with high probability, every eigenvector of a random matrix is delocalized in the sense that any subset of its coordinates carries a non-negligible portion of its $\ell_2$ norm. Our results pertain to a wide class of random matrices, including matrices with independent entries, symmetric and skew-symmetric matrices, as well as some other naturally arising ensembles. The matrices can be real and complex; in the latter case we assume
that the real and imaginary parts of the entries are independent.
\end{abstract}

\setcounter{tocdepth}{1}
\tableofcontents

\section{Introduction}

While eigenvalues of random matrices have been extensively studied since 1950-s
(see \cite{AGZ, BS, Tao} for introduction), less is known about
eigenvectors of random matrices. For matrices whose distributions are invariant under unitary or orthogonal
transformations, the picture is trivial: their normalized eigenvectors are uniformly distributed over
the unit Euclidean sphere.
Examples of such random matrices include the classical Gaussian Unitary Ensemble (GUE),
Gaussian Orthogonal Ensemble (GOE) and Ginibre ensembles.
All entries of these matrices are normal, and either all of them are independent (in Ginibre ensemble)
or independence holds modulo symmetry (in GUE and GOE).

Guided by the ubiquitous universality phenomenon in random matrix theory (see \cite{Tao-Vu BullAMS, Tao-Vu survey, EY BullAMS, Erdos}), we can
anticipate that the eigenvectors behave in a similar way for a much broader class of random matrices.
Thus, for a general $n \times n$ random matrix $A$ with independent entries we may expect that the normalized eigenvectors are approximately uniformly distributed on the unit sphere. The same should
hold for a general {\em Wigner matrix} $A$, a symmetric random matrix with independent entries
on and above the diagonal.

The uniform distribution on the unit sphere has several remarkable properties. Showing that
the eigenvectors of general random matrices have these properties, too, became a focus
of attention in the recent years \cite{ESY 1, ESY 2, EK 1, EK 2, BG, EKYY I, EKYY II, EKYY band, TVW, VW, BP, CMS, RV delocalization}.
One of such properties is {\em delocalization in the sup-norm}.
For a random vector $v$ uniformly distributed on the unit sphere, a quick check reveals that
no coefficients can be too large; in particular $\|v\|_\infty = O(\sqrt{\log n}/{\sqrt{n}})$ holds
with high probability. Establishing a similar delocalization property for eigenvectors of random matrices
is a challenging task. For eigenvectors of Hermitian random matrices, a weaker bound
$\|v\|_\infty = O(\log^\gamma n/{\sqrt{n}})$, with $\gamma = O(1)$,
was shown by Erd\"os et. al. \cite{ESY 1, ESY 2} using spectral methods.
Later, Vu and Wang \cite{VW} obtained the optimal exponent $\g=1/2$ for most eigenvectors (those
corresponding to the bulk of the spectrum).
Recently, the authors of the current paper established delocalization for random matrices with
all independent entries by developing a completely different, geometric approach \cite{RV delocalization}.

\subsection{No-gaps delocalization}

In the present paper, we will address a different natural delocalization property.
Examining a random vector uniformly distributed on the sphere, we may notice
that its mass (the $\ell_2$ norm) is more or less evenly spread over the coordinates.
There are no ``gaps'' in the sense that all all subsets $J \subset [n]$ carry a non-negligible
portion of the mass.

The goal of this paper is to establish this property for eigenvectors of random matrices.
Formally, we would like to show that with high probability, for  any eigenvector $v$,
any $\e \in (0,1)$, and any subset of coordinates $J \subset [n]$ of size at least $\e n$, one has
\[
  \Big( \sum_{j \in J} |v_j|^2 \Big)^{1/2}
  \ge \phi(\e) \|v\|_2,
\]
where $\phi: (0,1) \to (0,1)$ is some nice function.
We call this phenomenon {\em no-gaps delocalization}.

One may wonder about the relation of the no-gaps delocalization to the delocalization in the sup-norm
we mentioned before. As is easy to see, neither of these two properties implies the other.
They offer complementary insights into the behavior of the coefficients of the eigenvectors --
one property rules out peaks and the other rules out gaps.

The need for no-gaps delocalization arises naturally in problems of spectral graph theory. A similar notion appeared in the pioneering work of Dekel et. al. \cite{DLL}. The desirability of establishing no-gaps delocalization was emphasized in a paper of Arora and Bhaskara \cite{AB}, where a similar but weaker property was proved for a fixed subset $J$. Very recently, Eldan et. al. \cite{ERS} established a weaker form of no-gaps delocalization for the Laplacian of an Erd\"os-R\'enyi graph with $\e>1/2$ and the function $\phi$ depending on $\e$ and $n$. This delocalization has been used to prove a version of a conjecture of Chung  on the influence of adding or deleting edges on the spectral gap of an Erd\"os-R\'enyi graph. For shifted Wigner matrices and one-element sets $J$, the no-gaps delocalization was proved by Nguyen et. al. \cite{NTV} with $\phi(1/n)=(1/n)^{C}$ for some absolute constant $C$.

\medskip

In the present paper, we prove the no-gaps delocalization for a wide set of ensembles of random matrices including matrices with independent entries, symmetric and skew-symmetric random matrices, and others. Explicitly, we make the following assumption about possible dependencies among the entires.

\begin{assumption}[Dependences of entries]         \label{A}
  Let $A$ be an $N \times n$ random matrix.
  Assume that for any $i,j \in [n]$, the entry $A_{ij}$ is independent of the rest of the entries except possibly $A_{ji}$.
We also assume that the real part of $A$ is random and the imaginary part is fixed.
\end{assumption}

Note that Assumtion~\ref{A} implies the following important independence property, which we will repeatedly
use later: for any $J \subset [N]$, the entries of the submatrix $A_{J \times J^c}$ are independent.

Fixing the imaginary part in Assumtion~\ref{A} allows us to handle real random matrices.
This assumption can also be arranged for complex matrices with independent real and imaginary parts,
once we condition on the imaginary part. One can even consider a more general situation
where the real parts of the entries conditioned on the imaginary parts have variances bounded below.

We will also assume $\|A\| = O(\sqrt{n})$ with high probability.
This natural condition holds, in particular, if the entries of $A$ have mean zero and
bounded fourth moments \cite{Latala}.
To make this rigorous, we fix a number $M \ge 1$
and introduce the boundedness event
\begin{equation}         \label{eq: BAM}
\BB_{A,M} := \left\{ \|A\| \le M \sqrt{n} \right\}.
\end{equation}

\subsection{Main results}

Let us start with the simpler case where matrix entries have continuous distributions.
This will allow us to present the method in the most transparent way,
without having to navigate numerous obstacles that arise for discrete distributions.

\begin{assumption}[Continuous distributions]	\label{A: continuous distribution}
  We assume that the real parts of the matrix entries have
  densities bounded by some number $K \ge 1$.
\end{assumption}

Under Assumptions \ref{A} and \ref{A: continuous distribution}, we show that every subset of at least
eight coordinates carries a non-negligible part of the mass of any eigenvector.
This is summarized in the following theorem.

\begin{theorem}[Delocalization: continuous distributions]	\label{thm: delocalization continuous}
  Let $A$ be an $n \times n$ random matrix which satisfies Assumptions \ref{A} and \ref{A: continuous distribution}.
  Choose $M \ge 1$ such that the boundedness event $\BB_{A,M}$ holds with probability at least $1/2$.
  Let $\e \in (8/n,1/2)$ and $s>0$.
  Then, conditionally on $\BB_{A,M}$, the following holds with probability
  at least $1-(cs)^{ \e n}$.
  Every eigenvector $v$ of $A$ satisfies
  $$
  \|v_I\|_2 \ge (\e s)^6 \|v\|_2 \quad \text{for all } I \subset [n], \; |I| \ge \e n.
  $$
  Here $c = c(K,M)>0$.
\end{theorem}

The restriction $\e<1/2$ can be easily removed, see Remark \ref{rem: large epsilon} below.

Note that we do not require any moments for the matrix entries, so heavy-tailed distributions are allowed.
However, the boundedness assumption formalized by \eqref{eq: BAM}
implicitly yields some upper bound on the tails. Indeed, if the entries of $A$
are i.i.d. and mean zero, then $\|A\| = O(\sqrt{n})$ can only hold if the
fourth moments of entries are bounded \cite{BSY}.

Further, we do not require that the entries of $A$ have mean zero.
Therefore, adding to $A$ any fixed matrix of norm $O(\sqrt{n})$ does not affect our results.

\medskip

Extending Theorem \ref{thm: delocalization continuous} to general, possibly discrete distributions,
is a challenging task. We are able to do this for matrices with identically distributed entries
and under the mild assumption that the distributions of entries are not too concentrated near a single number.

\begin{assumption}[General distribution of entries]	\label{A: general distribution}
  We assume that the real parts of the matrix entries are
  i.i.d. copies of a random variable $\xi$, which satisfies
  \begin{equation}							\label{eq: xi}
  \sup_{u \in \R} \Pr{|\xi-u| \le 1} \le 1-p, \quad \Pr{|\xi| > K} \le p/2 \quad \text{for some } K, p >0.
  \end{equation}
\end{assumption}

Among many examples of discrete random variables $\xi$ satisfying Assumption~\ref{A: general distribution},
the most prominent one is the symmetric {\em Bernoulli} random variable $\xi$,
which takes values $-1$ and $1$ with probability $1/2$ each.

With Assumption \ref{A: continuous distribution} replaced by \ref{A: general distribution}, we can prove the no-gaps delocalization result, which we summarize in the following theorem.

\begin{theorem}[Delocalization: general distributions]	\label{thm: delocalization general}
  Let $A$ be an $n \times n$ random matrix which satisfies Assumptions \ref{A} and \ref{A: general distribution}.
  Choose $M \ge 1$ such that the boundedness event $\BB_{A,M}$ holds with probability at least $1/2$.
  Let $\e \ge 1/n$ and $s \ge c_1 \e^{-7/6} n^{-1/6} +e^{-c_2/\sqrt{\e}}$.
  Then, conditionally on $\BB_{A,M}$, the following holds with probability
  at least $1-(c_3 s)^{\e n}$.
  Every eigenvector $v$ of $A$ satisfies
  $$
  \|v_I\|_2 \ge (\e s)^6 \|v\|_2 \quad \text{for all } I \subset [n], \; |I| \ge \e n.
  $$
  Here $c_k = c_k(p,K,M)>0$ for $k=1, 2, 3$.
\end{theorem}

\begin{remark}  \label{rem: large epsilon}
  The restriction $s< 1/c_3$ making the theorem meaningful implies that
  $\e \in (c_4 n^{-1/7},c_5)$ for some $c_4>0$ and $c_5<1$.
  The upper bound, however, can be easily removed.
  If $\e \ge c_5$, then delocalization event
  $$
  \|v_I\|_2 \ge c_6 \|v\|_2 \quad \text{for all } I \subset [n], \; |I| \ge \e n
  $$
  holds with probability at least $1-e^{-c_7 n}$.
  This follows by applying Theorem \ref{thm: delocalization general}
  with a sufficiently small constant $\e=c_7$ which would allow to choose $s=e^{-1} c_3$.
\end{remark}

The restrictions on $\e$ and $s$ can be significantly relaxed; see the end of Section \ref{s: invertibility general}.
We did not attempt to optimize these bounds, striving for clarity of the argument in lieu of more precise estimates.

\section{Outline of the argument}  \label{s: outline}

Our approach to Theorems~\ref{thm: delocalization continuous} and \ref{thm: delocalization general}
is based on reducing delocalization to invertibility of random matrices.
We will now informally explain this reduction, which is quite flexible and can be applied
for many classes of random matrices.

\subsection{Reduction of delocalization to invertibility}				\label{s: reduction intro}

Let us argue by contradiction.
Suppose there exists a {\em localized} unit eigenvector $v$ of $A$, which means that
\begin{equation}         \label{eq: v localized outline}
\|v_I\|_2 = o(1) \quad \text{for some index subset } I \subset [n], \; |I| = \e n.
\end{equation}
Let us decompose the matrix\footnote{For convenience of notation, we skip the identity symbol
thus writing $A-\l$ for $A-\l I$.} $B := A-\l$ into two sub-matrices, $B_I$ that consists of columns
indexed by $I$ and $B_{I^c}$ with columns indexed by $I^c$. Then
\begin{equation}         \label{eq: 0 decomposed outline}
0 = Bv = B_I v_I + B_{I^c} v_{I^c}.
\end{equation}
To estimate the norm of $B_I v_I$, note that the operator norm of $B$ can be bounded as
$$
\|B_I\| \le \|B\| \le 2\|A\| = O(\sqrt{n}) \quad \text{with high probability},
$$
where we used the boundedness event \eqref{eq: BAM}.
Combining with \eqref{eq: v localized outline}, we obtain
$$
\|B_I v_I\|_2 = o(\sqrt{n}).
$$
But the identity \eqref{eq: 0 decomposed outline} implies that the norms of $B_I v_I$ and $B_{I^c} v_{I^c}$
are the same, thus
\begin{equation}         \label{eq: B not invertible outline}
\|B_{I^c} v_{I^c}\|_2 = o(\sqrt{n}).
\end{equation}
Since $v$ is a unit vector and $v_I$ has a small norm, the norm of $v_{I^c}$ is close to $1$.
Then \eqref{eq: B not invertible outline} implies that the matrix $B_{I^c}$ is {\em not well invertible} on its range.
Formally, this can be expressed as a bound on the smallest singular value:
\begin{equation}         \label{eq: smin small outline}
\smin(B_{I^c}) = o(\sqrt{n}).
\end{equation}
Recall that $B_{I^c}$ is an $n \times (n-\e n)$ random matrix. Thus we reduced delocalization to
quantitative invertibility of almost square random matrices.

\subsection{Invertibility of random matrices}

A standard expectation in the non-asymptotic random matrix theory is that random matrices
are well invertible, and the bad event \eqref{eq: smin small outline} should not hold.
For example, if the $n \times (n-\e n)$ random matrix $H = B_{I^c}$ had
all independent standard normal entries, then we would have the desired lower bound
\begin{equation}         \label{eq: Gaussian outline}
\smin(H) = \Omega(\sqrt{n})	\quad \text{with high probability},
\end{equation}
see e.g. \cite{V RMT}.
Invertibility results similar to \eqref{eq: Gaussian outline} are now available
for distributions more general than Gaussian (see \cite{RV rectangular}), and in particular
for discrete distributions. Handling discrete distributions in the invertibility problems
is considerably more challenging than discrete ones. Recent successes in these problems
were based on understanding the interaction of probability with {\em arithmetic structure}, which
was quantified via generalized arithmetic progressions in \cite{Tao-Vu Annals, Tao-Vu STOC} and approximate least common
denominators (LCD) in \cite{RV square, RV rectangular, V symmetric};
see \cite{Tao-Vu BullAMS, RV ICM} for background and references.

Nevertheless, there are significant difficulties in our situation that prevent us from deducing
\eqref{eq: Gaussian outline} for $H = B_{I^c}$ from any previous work. Let us mention
some of these difficulties.

\subsubsection{Lack of independence}
Not all entries of $A$ (and thus of $H$) may be independent.
    As we recall from Assumption~\ref{A}, we are looking for ways to control symmetric and
    non-symmetric matrices simultaneously.
    This makes it necessary to extract rectangular blocks of independent entries from matrix $H$ and modify the definition of the LCD adapting it to this block extraction.

\subsubsection{Small exceptional probability required} We need that the delocalization result, and thus
    the invertibility bound \eqref{eq: Gaussian outline}, hold uniformly over all index subsets $I$\
    of size $\e n$.
    Since there are $\binom{n}{\e n} \sim \e^{-\e n}$ such sets, we would need
    the probability of non-invertibility \eqref{eq: smin small outline} to be at most $\e^{\e n}$.
    While this is possible to achieve for real matrices with all independent entries \cite{RV rectangular},
    such small exceptional probabilities (smaller than $e^{-\e n}$)
    may not come automatically for the general case.

\subsubsection{Complex entries}
    Results of the type \eqref{eq: Gaussian outline} which hold with the probability we need are
    available only for real matrices; see in particular \cite{RV ICM, V RMT, R survey}.
    Since eigenvalues $\l$ even of real matrices may be complex,
    we must work with complex random matrices. Extending the known results to complex matrices
    is non-trivial. Indeed, in order to preserve the
    matrix-vector multiplication, we replace a complex $n \times N$ random matrix $B = R + iT$
    by the real $2n \times 2N$ random matrix $\left[ \begin{smallmatrix} R & -T \\ T & R \end{smallmatrix} \right]$.
   The real and imaginary parts $R$ and $T$ each appear twice in this matrix,
   which causes extra dependences of the entries.
   Moreover, we encounter a major problem while trying to apply the covering argument to show that the least common denominator of the subspace orthogonal to a certain set of columns of $H$ is large.
   Indeed, since we have to consider a real $2n \times 2N$ matrix,
   we will have to construct a net in a subset of the real sphere of dimension $2N$. The size of such net is exponential in the dimension. On the other hand, the number of independent rows of $R$ is only $n$, so the small ball probability will be exponential in terms of $n$. As $n<N$, the union bound would not be applicable.

   To overcome this difficulty, we introduce a stratification of the complex sphere, partitioning it according to the correlation between the real and the imaginary parts of vectors.
   This stratification, combined with a modified definition of the least common denominator, allows us to obtain stronger small ball probability estimates for weakly correlated vectors in Section \ref{s: SBP via correlations}.
   Yet the set of weakly correlated vectors has a larger complexity, which is expressed in the size of the nets. The cardinality of the nets has to be accurately estimated in Section \ref{s: net}.
   These two effects, the improvement of the small ball probability estimate and the increase of the complexity, work against each other. In Section \ref{s: distance proof}, we show that they exactly balance each other, making it possible to apply the union bound.

\subsection{Organization of the argument}

After discussing basic background material in Section~\ref{s: notation-preliminaries},
we present a formal reduction of delocalization to invertibility in Section~\ref{s: reduction}.
The rest of the paper will focus on invertibility
of random matrices. Section~\ref{s: continuous} covers continuous distributions; the main result
there is Invertibility Theorem~\ref{thm: invertibility continuous}, from which we quickly deduce
Delocalization Theorem~\ref{thm: delocalization continuous}.
These sections are relatively simple and can be read independently of the rest of the paper.

Invertibility of random matrices with general distributions is considerably more difficult.
We address this problem in Sections~\ref{s: invertibility general} -- \ref{s: general proof}.
The main result there is Invertibility Theorem~\ref{thm: invertibility general}, from which we quickly deduce
Delocalization Theorem~\ref{thm: delocalization general}.

Our general approach to invertibility follows the method developed by the authors in \cite{RV square, RV rectangular},
see also \cite{RV ICM}.
We reduce proving invertibility to the {\em distance problem}, where we seek a lower bound
on $\dist(Z,E)$ where $Z$ is a random vector with independent coordinates
and $E$ is an independent random subspace in $\R^N$.
If we choose $E$ to be a hyperplane (subspace of codimension one),
we obtain an important class of examples in the distance problem, namely
{\em sums of independent random variables}.

In Section~\ref{s: LCD and SBP} we study small ball probabilities for sums of real-valued independent
random variables, as well as their higher dimensional versions $\dist(Z,E)$. These probabilities
are controlled by the {\em arithmetic structure} of $E^\perp$, which we quantify via so-called {\em least
common denominator} (LCD) of $E^\perp$. The larger LCD, the more $E^\perp$ is arithmetically
unstructured, and the better are small ball probabilities for $\dist(Z,E)$.
We formalize this relation in the very general Theorem~\ref{thm: SBP LCD},
and then we specialize in Sections~\ref{s: special cases sums} and \ref{s: special cases projections}
to sums of independent random variables and distances to subspaces.

In Section~\ref{s: distance statement}, we state our main bound on the distance between
random vectors and subspaces; this is Theorem~\ref{thm: distance general}. In order to deduce
this result from the small ball probability bounds of Section~\ref{s: LCD and SBP}, two things need to be done:
(a) transfer the problem from complex to real, and (b) show that random subspaces are arithmetically
unstructured, i.e. the LCD of $E^\perp$ is large. The transfer to a real problem
is done in Section~\ref{s: from C to R general}, and then our main focus becomes the structure of subspaces.

By the nature of our problem, the subspaces $E^\perp$ will be the kernels of random matrices.
The analysis of such kernels starts in Section~\ref{s: kernels incompressible}. We show there that
all vectors in $E^\perp$ are {\em incompressible}, which means that they are not localized on
a small fraction of coordinates.

Unfortunately, in the process of transferring the problem from complex to real in
Section~\ref{s: from C to R general} introduces extra dependences among the entries of the random matrix.
In Section~\ref{s: SBP via correlations} we adjust our results on small ball probabilities so they are not
destroyed by those dependences. We find that these probabilities are controlled not only on LCD
but also by {\em real-imaginary correlations} of the vectors in $E^\perp$.

Recall that our goal is to show that all vectors in $E^\perp = \ker(B)$ are unstructured, i.e. they have
large LCD. We would obtain this if we can lower-bound $Bz$ for all vectors with small LCD.
For a fixed $z$, a lower bound follows from the small ball probability results of Section~\ref{s: SBP via correlations}. To make the bound uniform, it is enough to run a union bound over a good net of
the set of vectors with small LCD.
We construct a good {\em net for level sets of LCD and real-imaginary correlations} in Section~\ref{s: net}.
Informally, small LCD or small correlation impose strong constraints, which make it possible to
construct a smaller net than based on the trivial (volume-based) argument.

After this major step, the argument can be wrapped up relatively easily.
In Section~\ref{s: distance proof} we finalize the distance problem.
We combine the small ball probability results with the fact that the random subspace
are unstructured, and deduce Theorem~\ref{thm: distance general}.

In Section~\ref{s: general proof} we finalize the invertibility problem for general distributions;
here we deduce Theorem~\ref{thm: invertibility general}. This is done by modifying the argument
for continuous distributions in Section~\ref{s: continuous} using the non-trivial distance bound
Theorem~\ref{thm: distance general} for general distributions.

\section{Notation and preliminaries}			\label{s: notation-preliminaries}
Throughout the paper, by $C, c, C_1, \ldots$ we denote constants that may depend only on the parameters
$K$ and $p$ that control the distributions of matrix entries in Assumptions~\ref{A: continuous distribution}
and \ref{A: general distribution} and the parameter $M$ that controls the matrix norm in \eqref{eq: BAM}.

We denote by $S_{\R}^{n-1}$ and $S_{\C}^{n-1}$ the unit spheres of $\R^n$ and $\C^n$ respectively.
We denote by $B(a,r)$ the Euclidean ball in $\R^n$ centered at a point $a$ and with radius $r$.
The unit sphere of a subspace $E$ will be denoted $S_E$, and
the orthogonal projection onto a subspace $E$ by $P_E$.

Given an $n \times m$ matrix $A$
and an index sets $J \subset [n]$, by $A_J$ we denote the $n \times |J|$ sub-matrix of $A$
obtained by including the columns indexed by $J$. Similarly, for a vector $z \in \C^n$,
by $z_J$ we denote the vector in $\C^J$ which consists of the coefficients indexed by $J$.

\subsection{Concentration function}

The concept of {\em concentration function} has been introduced by P.~L\'evy and
studied in probability theory for several decades, see \cite{RV ICM} for the classical and recent history.

\begin{definition}[Concentration function] \label{def: concentration function}
  Let $Z$ be a random vector taking values in $\C^n$. The concentration function of $X$
  is defined as
  $$
  \LL(Z,t) = \Pr{\|Z-u\|_2 \le t}, \quad t \ge 0.
  $$
\end{definition}

The concentration function gives a uniform upper bound on the {\em small ball probabilities} for $X$.
We defer a detailed study of concentration function for sums of independent random variables to Section~\ref{s: LCD and SBP}.
Let us mention here only one elementary restriction property.
\begin{lemma}[Small ball probabilities: restriction]				\label{lem: SBP restriction}
  Let $\xi_1, \ldots, \xi_N$ be independent random variables and $a_1,\ldots,a_N$ be real numbers.
  Then, for every subset of indices $J \subset [N]$ and every $t \ge 0$ we have
  $$
  \LL \Big( \sum_{j \in J} a_j \xi_j, t \Big) \le \LL \Big( \sum_{j=1}^N a_j \xi_j, t \Big).
  $$
\end{lemma}
\begin{proof}
This bound follows easily by conditioning on the random variables $\xi_j$ with $j \not\in J$
and absorbing their contribution into a fixed vector $u$ in the definition of the concentration function.
\end{proof}

We will also use a simple and useful tensorization property which goes back to \cite{V square, RV square}.

\begin{lemma}[Tensorization]				\label{lem: tensorization}
  Let $Z = (Z_1,\ldots,Z_n)$ be a random vector in $\C^n$ with independent coordinates.
  Assume that there exists numbers $t_0, M \ge 0$ such that
  $$
  \LL(Z_j, t) \le M(t+t_0) \quad \text{for all $j$ and $t \ge 0$.}
  $$
  Then
  $$
  \LL(Z,t\sqrt{n}) \le \left[ CM(t+t_0) \right]^n \quad \text{for all $t \ge 0$.}
  $$
\end{lemma}

\begin{proof}
By translation, we can assume without loss of generality that $u=0$ in the definition
of concentration function. Thus we want to bound the probability
$$
\Pr{\|Z\|_2 \le t \sqrt{n}}  = \Pr{ \sum_{j=1}^n |Z_j|^2 \le t^2 n }.
$$
Rearranging the terms, using Markov's inequality and then independence,
we can bound this probability by
\begin{equation}         \label{eq: SBP via MGF}
\Pr{ n - \frac{1}{t^2} \sum_{j=1}^n |Z_j|^2 > 0 }
\le \E \exp \left( n - \frac{1}{t^2} \sum_{j=1}^n |Z_j|^2 \right)
= e^n \prod_{j=1}^n \E \exp(-|Z_j|^2/t^2).
\end{equation}
To bound each expectation, we use the distribution integral formula followed by a change of variables.
Thus
$$
\E \exp(-|Z_j|^2/t^2)
= \int_0^1 \Pr{\exp(-|Z_j|^2/t^2) > x} dx
= \int_0^\infty 2 y e^{-y^2} \, \Pr{|Z_j| < t y} dy.
$$
By assumption, we have $\Pr{|Z_j| < t y} \le M(ty+t_0)$. Substituting this into the integral
and evaluating it, we obtain
$$
\E \exp(-|Z_j|^2/t^2) \le CM(t+t_0).
$$
Finally, substituting this into \eqref{eq: SBP via MGF}, we see that the probability
in question is bounded by $e^n [CM(t+t_0)]^n$. This completes the proof of the lemma.
\end{proof}

\section{Reduction of delocalization to invertibility of random matrices}				\label{s: reduction}
In this section, we show how to deduce delocalization from
quantitative invertibility of random matrices. We outlined this reduction in Section~\ref{s: reduction intro}
and will now make it formal.
For simplicity of notation, we shall assume that $\e n/2 \in \N$, and
we introduce the {\em localization} event
$$
\Loc(A, \e, \d) := \left\{ \exists \text{ eigenvector } v \in S_{\C}^{n-1}, \, \exists I \subset [n], \; |I| = \e n:
\|v_I\|_2 < \d \right\}.
$$
Since we assume in Theorem~\ref{thm: delocalization continuous} that the boundedness
event $\BB_{A,M}$ holds with probability at least $1/2$, the conclusion of that theorem can stated as follows:
\begin{equation}         \label{eq: PLB}
\Pr{ \Loc(A, \e, (\e s)^6) \text{ and } \BB_{A,M} } \le (cs)^{ \e  n}.
\end{equation}
The following proposition reduces proving delocalization results like \eqref{eq: PLB} to an invertibility bound.

\begin{proposition}[Reduction of delocalization to invertibility]			\label{prop: reduction}
  Let $A$ be an $n \times n$ random matrix with arbitrary distribution.
  Let $M \ge 1$ and $\e, p_0, \d \in (0,1/2)$.
  Assume that for any number $\l_0 \in \C$, $|\l_0| \le M \sqrt{n}$,
  and for any set $I \subset [n]$, $|I| = \e n$, we have
  \begin{equation}         \label{eq: invertibility assumption}
  \Pr{\smin \big( (A - \l_0)_{I^c} \big) \le 8 \d M \sqrt{n} \text{ and } \BB_{A,M} } \le p_0.
  \end{equation}
  Then
  $$
  \Pr{ \Loc(A, \e, \d) \text{ and } \BB_{A,M} } \le 5\d^{-2} (e/\e)^{\e n} p_0.
  $$
\end{proposition}

\begin{proof}
Assume both the localization event and the boundedness event $\BB_{A,M}$ hold.
Using the definition of $\Loc(A, \e, \d)$, choose a localized eigenvalue-eigenvector pair $(\l, v)$
and an index subset $I$.
Decomposing the eigenvector as
$$
v = v_I + v_{I^c}
$$
and multiplying it by $A-\l$, we obtain
$$
0 = (A-\l) v = (A-\l)_I v_I + (A-\l)_{I^c} v_{I^c}.
$$
By triangle inequality, this yields
$$
\|(A-\l)_{I^c} v_{I^c}\|_2
= \|(A-\l)_I v_I\|_2
\le (\|A\| + |\l|) \|v_I\|_2.
$$
By the localization event $\Loc(A, \e, \d)$, we have $\|v_I\|_2 \le \d$.
By the boundedness event $\BB_{A,M}$ and since $\l$ is an eigenvalue of $A$,
we have $|\l| \le \|A\| \le M \sqrt{n}$. Therefore
\begin{equation}         \label{eq: A-l v}
\|(A-\l)_{I^c} v_{I^c}\|_2 \le 2M \d \sqrt{n}.
\end{equation}

This happens for some $\lambda$ in the disc $\{z \in \C: |z| \le M \sqrt{n}\}$.
We will now run a covering argument in order to fix $\lambda$.
Let $\NN$ be a $(2M \d \sqrt{n})$-net of that disc.
One can construct $\NN$ so that
$$
|\NN| \le \frac{5}{\d^2}.
$$
Choose $\l_0 \in \NN$ so that $|\l_0 - \l| \le 2M \d \sqrt{n}$. By \eqref{eq: A-l v}, we have
\begin{equation}         \label{eq: A-l0 v}
\|(A-\l_0)_{I^c} v_{I^c}\|_2 \le 4M \d \sqrt{n}.
\end{equation}
Since $\|v_I\|_2 \le \d \le 1/2$, we have $\|v_{I^c}\|_2 \ge \|v\|_2 - \|v_I\|_2 \ge 1/2$.
Therefore, \eqref{eq: A-l0 v} implies that
\begin{equation}         \label{eq: smin A-l0}
\smin((A - \l_0)_{I^c}) \le 8M \d \sqrt{n}.
\end{equation}

Summarizing, we have shown that the events $\Loc(A, \e, \d)$ and $\BB_{A,M}$ imply the
existence of a subset $I \subset [n]$, $|I| = \e n$, and a number $\l_0 \in \NN$, such that
\eqref{eq: smin A-l0} holds.
Furthermore, for fixed $I$ and $\l_0$, assumption \eqref{eq: invertibility assumption} states that
\eqref{eq: smin A-l0} together with $\BB_{A,M}$ hold with probability at most $p_0$.
So by the union bound we conclude that
$$
\Pr{ \Loc(A, \e, \d) \text{ and } \BB_{A,M} } \le \binom{n}{\e n} \cdot |\NN| \cdot p_0
\le \Big(\frac{e}{\e} \Big)^{\e n} \cdot \frac{5}{\d^2} \cdot p_0.
$$
This completes the proof of the proposition.
\end{proof}

\section{Invertibility for continuous distributions}		\label{s: continuous}

The reduction we made in the previous section puts invertibility of random matrices into the spotlight.
Our goal becomes to establish invertibility property \eqref{eq: invertibility assumption}.
In this section we do this for matrices for continuous distributions.

\begin{theorem}[Invertibility: continuous distributions]			\label{thm: invertibility continuous}
     Let $A$ be an $n \times n$ random matrix satisfying the assumptions of Theorem~\ref{thm: delocalization continuous}.
     Let $M \ge 1$, $\e \in (0,1)$, and let $I \subset [n]$ be any fixed subset with $|I| = \e n$.
     Then for any $t>0$, we have
     $$
     \Pr{ \smin(A_{I^c}) \le t \sqrt{n} \text{ and } \BB_{A,M} } \le (C K M t^{0.4} \e^{-1.4})^{\e n/2}.
     $$
\end{theorem}

Before we pass to the proof of this result, let us first see how it implies delocalization.

\subsection{Deduction of Delocalization Theorem~\ref{thm: delocalization continuous}}	
\label{s: delocalization continuous}

Let $A$ be a matrix as in Theorem~\ref{thm: delocalization continuous}.
We are going to use Proposition~\ref{prop: reduction}, so let us choose
$\l_0$ and $I$ as in that proposition and try to check the invertibility
condition \eqref{eq: invertibility assumption}.
Observe that the shifted matrix $A-\l_0$ still satisfies the assumptions of
Theorem~\ref{thm: invertibility continuous}, and $\BB_{A, M}$ implies $\BB_{A-\l_0, 2M}$
because $|\l_0| \le M \sqrt{n}$. So we can apply Theorem~\ref{thm: invertibility continuous}
for $A-\l_0$ and with $2M$, which yields
$$
\Pr{\smin(A-\l_0 \, \text{Id})_{I^c} \le t \sqrt{n}
\text{ and } \BB_{A,M}} \le p_0,
$$
for $t \ge 0$, where $p_0 =  (C K M t^{0.4} \e^{-1.4})^{\e n/2}$.
Therefore invertibility condition \eqref{eq: invertibility assumption} holds for
$\d = t/8M$ and $p_0$.
Applying Proposition~\ref{prop: reduction}, we conclude that
$$
\Pr{ \Loc(A, \e, t/8M) \text{ and } \BB_{A,M} }
  \le 5 \Big(\frac{8M}{t}\Big)^2 (e/\e)^{\e n} p_0
$$
for $t \ge 0$. Setting $t = 8M (\e s)^6$ and substituting the value of $p_0$, we obtain
$$
\Pr{ \Loc(A, \e, (\e s)^6 ) \text{ and } \BB_{A,M} } \le (C(K,M) s)^{\e n}
$$
for $s \ge 0$. This completes the proof of Theorem~\ref{thm: delocalization continuous}.
\qed

\medskip

The proof of Theorem~\ref{thm: invertibility continuous} will occupy the rest of this section.

\subsection{Decomposition of the matrix}				\label{s: decomposition of matrix}

To make the proof of Theorem~\ref{thm: invertibility continuous} more convenient,
let us change notation slightly, namely replace $\e$ with $2\e$. Thus $A$ is an
$(1+2\e)n \times (1+2\e)n$ matrix and $|I| = 2\e n$.
The desired conclusion then would change to
\begin{equation}							\label{eq: invertibility continuous simplified}
\Pr{\smin(A_{I^c}) \le t \sqrt{n} \text{ and } \BB_{A,M}} \le (C K M t^{0.4} \e^{-1.4})^{\e n}.
\end{equation}
Without loss of generality, we can assume that $I$ is the interval of the {\em last} $2 \e n$ indices.

Let us decompose $A_{I^c}$ as follows:
\begin{equation}         \label{eq: A decomposed}
\bar{A} : = A_{I^c} =
\begin{bmatrix}
\phantom{X} B \phantom{X} \\ \phantom{X} G \phantom{X}
\end{bmatrix},
\end{equation}
where $B$ and $G$ are rectangular matrices of size $(1+\e)n \times n$ and $\e n \times n$ respectively.
By Assumption \ref{A}, the random matrices $B$ and $G$ are independent, and moreover
all entries of $G$ are independent.

We are going to show that either $\|Bx\|_2$ or $\|Gx\|_2$ is nicely bounded below for every vector $x \in S_{\C}^{n-1}$.
To control $B$, we use the second negative moment identity
to bound the Hilbert-Schmidt norm of the pseudo-inverse of $B$.
We deduce from it that most singular values of $B$ are not too small --
namely, all but $0.1 \e n$ singular values are bounded below by $\gtrsim \sqrt{\e n}$.
It follows that $B$ is nicely bounded below when restricted onto a subspace
of codimension $0.1 \e n$. (This subspace is formed by the corresponding singular vectors.)
Next, we condition on $B$ and we use $G$ to control the remaining $0.1 \e n$ dimensions.
A simple covering argument shows that $G$ is nicely
bounded below when restricted to a subspace of dimension $0.1 \e n$.
Therefore, either $B$ or $G$ is nicely bounded below on the entire space, and thus
$A$ is nicely bounded below on the entire space as well.

We will now pass to a detailed proof of Theorem~\ref{thm: invertibility continuous}.

\subsection{Distances between random vectors and subspaces}

In this section we start working toward bounding $B$ below on a large subspace.
We quickly reduce this problem to a control of the distance between a random
vector (a column of $B$) and a random subspace (the span of the rest of the columns).
We then prove a lower bound for this distance.

\subsubsection{Negative second moment identity}

The negative second moment identity \cite[Lemma A.4]{Tao-Vu} expresses
the Hilbert-Schmidt norm of the pseudo-inverse of $B$ as follows:
$$
\sum_{j=1}^n s_j(B)^{-2} = \sum_{i=1}^n \dist(B_j, H_j)^{-2}
$$
where $s_j(B)$ denote the singular values of $B$, $B_j$ denote
the columns of $B$, and $H_j = \Span(B_k)_{k \ne j}$.

To bound the sum above, we will establish a lower bound on the distance
between the random vector $B_j \in \C^{(1+\e)n}$
and random subspace $H_j \subseteq \C^{(1+\e)n}$ of complex dimension $n-1$.

\subsubsection{Enforcing independence of vectors and subspaces}

Let us fix $j$. If all entries of $B$ are independent, then $B_j$ and $H_j$ are independent.
However, Assumption \ref{A} leaves a possibility for $B_j$ to be correlated with $j$-th
row of $B$. This means that $B_j$ and $H_j$ may be dependent, which would
complicate the distance computation.

There is a simple way to remove the dependence by projecting out the $j$-th coordinate.
Namely, let $B'_j \in \C^{(1+\e)n-1}$ denote the vector $B_j$ with $j$-th coordinate removed,
and let $H'_j = \Span(B'_k)_{k \ne j}$. We note the two key facts. First,
$B'_j$ and $H'_j$ are independent by Assumption \ref{A}.
Second,
\begin{equation}         \label{eq: BjHj}
\dist(B_j, H_j) \ge \dist(B'_j, H'_j),
\end{equation}
since the distance between
two vectors can only decrease after removing a coordinate.

Summarizing, we have
\begin{equation}         		\label{eq: second negative moment}
\sum_{j=1}^n s_j(B)^{-2} \ge \sum_{i=1}^n \dist(B'_j, H'_j)^{-2}.
\end{equation}
Recall that $B'_j \in \C^{(1+\e)n-1}$ is a random vector with independent
entries whose real parts have densities bounded by $K$
(by Assumptions~\ref{A} and \ref{A: continuous distribution});
and $H'_j$ is an independent subspace of $\C^{(1+\e)n-1}$ of complex dimension $n-1$.

We are looking for a lower bound for the distances $\dist(B'_j, H'_j)$.
It is convenient to represent them via the orthogonal projection of $B'_j$ onto $(H'_j)^\perp$:
\begin{equation}         \label{eq: dist as proj}
\dist(B'_j, H'_j) = \| P_{E_j} B'_j \|_2,
\quad \text{where} \quad
E_j = (H'_j)^\perp.
\end{equation}

\subsubsection{Transferring the problem from $\C$ to $\R$}				\label{s: from C to R}

We will now transfer the distance problem from the complex to the real field.
To this end, we define the operation $z \mapsto \real{z}$ that makes complex vectors real
in the obvious way:
$$
\text{for } z = x+iy \in \C^N, \text{ define } \real{z} = \vectwo{x}{y} \in \R^{2N}.
$$
Similarly, we can make a complex subspace $E \subset \C^N$ real by defining
$$
\real{E} = \{ \real{z} :\; z \in E \} \subset \R^{2N}.
$$
Note that this operation doubles the dimension of $E$.

Let us record two properties that follow straight from this definition.

\begin{lemma}[Elementary properties of operation $x \mapsto \real{x}$]				\label{lem: real}
  \begin{enumerate}[1.]
    \item \label{part: PE real}
      For a complex subspace $E$ and a vector $z$, one has
      $$
      \real{P_E z} = P_{\real{E}} \real{z}.
      $$
    \item \label{part: LL real}
      For a complex-valued random vector $X$ and $r \ge 0$, one has
      $$
      \LL(\real{X},r) = \LL(X,r).
      $$
  \end{enumerate}
\end{lemma}

Recall that the second part of this lemma is about the concentration function $\LL(X,r)$
we introduced in Section~\ref{s: notation-preliminaries}.

After applying the operation $z \mapsto \real{z}$ to the random vector $B'_j$ in \eqref{eq: second negative moment},
we encounter a problem. Since the imaginary part of $B'_j$ is fixed by Assumption~\ref{A},
only half of the coordinates of $\real{B'_j}$ will be random, and that will not be enough for us.
The following lemma solves this problem by randomizing all coordinates.

\begin{lemma}[Randomizing all coordinates]					\label{lem: Z real}
  Consider a random vector $Z=X+iY \in \C^N$ whose imaginary part $Y \in \R^N$ is fixed.
  Set $\widehat{Z} = \vectwo{X_1}{X_2} \in \R^{2N}$ where $X_1$ and $X_2$ are independent copies of $X$.
  Let $E$ be a subspace of $\C^N$. Then
  $$
  \LL(P_E Z, r) \le \LL(P_{\real{E}} \widehat{Z}, 2r)^{1/2}, \quad r \ge 0.
  $$
\end{lemma}

\begin{proof}
Recalling the definition of the concentration function, in order to bound $\LL(P_E Z, r)$
we need to choose arbitrary $a \in \C^N$ and find a uniform bound on the probability
$$
p := \Pr{ \|P_E Z - a\|_2 \le r }.
$$
By assumption, the random vector $Z = X + iY$ has fixed imaginary part $Y$.
So it is convenient to express the probability as
$$
p = \Pr{ \|P_E X - b\|_2 \le r }
$$
where $b = a - P_E(iY)$ is fixed. Let us rewrite this identity
using independent copies $X_1$ and $X_2$ of $X$ as follows:
$$
p = \Pr{ \|P_E X_1 - b\|_2 \le r } = \Pr{ \|P_E (iX_2) - ib\|_2 \le r }.
$$
(The last equality follows trivially by multiplying by $i$ inside the norm.)
By independence of $X_1$ and $X_2$ and using triangle inequality, we obtain
\begin{align*}
p^2
  &= \Pr{ \|P_E X_1 - b\|_2 \le r \text{ and } \|P_E (iX_2) - ib\|_2 \le r } \\
  &\le \Pr{ \|P_E (X_1 + iX_2) - b - ib \|_2 \le 2r } \\
  &\le \LL(P_E (X_1 + iX_2), 2r).
\end{align*}
Further, using part~\ref{part: LL real} and then part~\ref{part: PE real} of Lemma~\ref{lem: real}, we
see that
$$
\LL(P_E (X_1 + iX_2), 2r) = \LL(P_{\real{E}} (\real{X_1 + iX_2}), 2r) = \LL(P_{\real{E}} \widehat{Z}, 2r).
$$
Thus we showed that $p^2 \le \LL(P_{\real{E}} \widehat{Z}, 2r)$ uniformly in $a$.
By definition of the concentration function, this completes the proof.
\end{proof}

\subsubsection{Bounding the distances below}

We are ready to control the distances appearing in \eqref{eq: dist as proj}.

\begin{lemma}[Distance between random vectors and subspaces]		\label{lem: distance}
  For every $j \in [n]$ and $\tau>0$, we have
  \begin{equation}         \label{eq: distance}
  \Pr{\dist(B'_j, H'_j) < \tau \sqrt{\e n}} \le (C K \tau)^{\e n}.
  \end{equation}
\end{lemma}

\begin{proof}
Representing the distances via projections of $B_j'$ onto the subspaces $E_j = (H'_j)^\perp$
as in \eqref{eq: dist as proj}, and using the definition of the concentration function, we have
$$
p_j := \Pr{\dist(B'_j, H'_j) < \tau \sqrt{\e n}}
\le \LL(P_{E_j} B'_j, \, \tau \sqrt{\e n}).
$$
Recall that $B_j'$ and $E_j$ are independent, and let us condition on $E_j$.
Lemma~\ref{lem: Z real} implies that
$$
p_j \le \LL(P_{\real{E_j}} \widehat{Z}, \, \tau \sqrt{\e n})^{1/2}
$$
where $\widehat{Z}$ is a random vector with independent coordinates
that have densities bounded by $K$.

Recall that $H'_j$ has codimension $\e n$; thus $E_j$ has dimension $\e n$
and $\real{E_j}$ has dimension $2\e n$.
We can use a bound on the small ball probability from \cite{RV small ball}, which states
that the density of $P_{\real{E_j}} \widehat{Z}$ is bounded by $(CK)^{2\e n}$.
Integrating the density over a ball of radius $2 \tau \sqrt{\e n}$ in the subspace $\real{E_j}$
that has volume $(C\tau)^{2\e n}$, we conclude that
$$
\LL(P_{\real{E_j}} \widehat{Z}, \, \tau \sqrt{\e n}) \le (C K \tau)^{2 \e n}.
$$
It follows that
$$
p_j \le (C K \tau)^{\e n},
$$
as claimed. The proof of Lemma~\ref{lem: distance} is complete.
\end{proof}

\subsection{$B$ is bounded below on a large subspace $E^+$}

\subsubsection{Plugging the distance bound into second moment inequality}		\label{s: dist into smi}

In order to substitute the bound \eqref{eq: distance} into the negative second moment
inequality \eqref{eq: second negative moment},
let us recall some classical facts about the weak $L^p$ norms. The weak $L^p$ norm
of a random variable $Y$ is defined as
\[
  \|Y\|_{p,\infty} = \sup_{t>0} \, t \cdot ( \Pr{|Y|>t} )^{1/p}.
\]
This is not a norm but is
equivalent to a norm if $p>1$. In particular, the weak triangle inequality holds:
\begin{equation}         \label{eq: weak Lp triangle}
\|\sum_i Y_i\|_{p,\infty} \le C(p) \sum_i \|Y_i\|_{p,\infty}
\end{equation}
where $C(p)$ is bounded above by an absolute constant for $p \ge 2$, see \cite{Stein Weiss}, Theorem 3.21.

The bound \eqref{eq: distance} means that $Y_i := \dist(B_i, H_i)^{-2}$ are in weak $L^p$
for $p = \e n/ 2$, and that $\|Y_i\|_{p,\infty} \le C^2K^2/\e n$. Since by assumption $p \ge 2$,
the weak triangle inequality  \eqref{eq: weak Lp triangle}
yields $\|\sum_{i=1}^n Y_i\|_{p,\infty} \le C_2^2K^2/\e$. This in turn means that
$$
\Pr{\sum_{i=1}^n \dist(B_i, H_i)^{-2} > \frac{1}{\tau^2 \e}} \le (C_2 K\tau)^{\e n}, \quad \tau > 0.
$$
Therefore, by the second negative moment identity \eqref{eq: second negative moment},
the event
\begin{equation}         \label{eq: sum moments small}
\EE_1 := \left\{ \sum_{i=1}^n s_i(B)^{-2} \le \frac{1}{\tau^2 \e} \right\}
\end{equation}
is likely: $\P((\EE_1)^c) \le (C_2 K \tau)^{\e n}$.

\subsubsection{A large subspace $E^+$ on which $B$ is bounded below}		\label{s: E+}

Fix a parameter $\tau>0$ for now, and assume that the event \eqref{eq: sum moments small} occurs.
By Markov's inequality, for any $\d >0$ we have
$$
\Big| \big\{ i: \; s_i(B) \le \d \sqrt{n} \big\} \Big|
= \Big| \big\{ i: \; s_i(B)^{-2} \ge \frac{1}{\d^2n} \big\} \Big|
\le \frac{\d^2 n}{\tau^2 \e}.
$$
Let $c \in (0,1)$ be a small absolute constant. Choosing $\d = c \tau \e$, we have
\begin{equation}         \label{eq: small singular values}
\Big| \big\{ i: \; s_i(B) \le c \tau \e \sqrt{n} \big\} \Big| \le c \e n.
\end{equation}
Let $v_i(B)$ be the right singular vectors of $B$, and
consider the (random) orthogonal decomposition
$\C^n = E^- + E^+$,
where
$$
E^- = \Span\{v_i(B): \; s_i(B) \le c \tau \e \sqrt{n}\}, \quad
E^+ = \Span\{v_i(B): \; s_i(B) > c \tau \e \sqrt{n}\}.
$$
Inequality \eqref{eq: small singular values} means that $\dim_{\C}(E^-) \le c \e n$.

Let us summarize. We obtained that the event
\begin{equation}							\label{eq: E- large}
\DD_{E^-} := \left\{ \dim(E^-) \le c \e n \right\}
\text{ satisfies }
\P((\DD_{E^-})^c) \le (C_2 K \tau)^{\e n},
\end{equation}
so $E^-$ is likely to be a small subspace and $E^+$ a large subspace.
Moreover, by definition, $B$ is nicely bounded below on $E^+$:
\begin{equation}							\label{eq: B bounded below}
\inf_{x \in S_{E^+}} \|Bx\|_2 \ge c \tau \e \sqrt{n}.
\end{equation}

\subsection{$G$ is bounded below on the small complementary subspace $E^-$}		\label{s: G below on E-}

Recall that the subspaces $E^+$ and $E^-$ are determined by the sub-matrix $B$,
so these subspaces are independent of $G$ by Assumption \ref{A}. Let us fix $B$ so that $\dim(E^-) \le c \e n$;
recall this is a likely event by \eqref{eq: E- large}.

Note that $G$ is an $\e n \times n$ random matrix with independent entries.
We are going to show that $G$ is well bounded below when restricted onto
the fixed subspace $E^-$. This can be done by a standard covering argument,
where a lower bound is first proved for a fixed vector, then extended to
a $\d$-net of the sphere by a union bound, and finally to the whole sphere
by approximation.

\subsubsection{Lower bounds on a fixed vector}				\label{s: lower on fixed vector}

\begin{lemma}[Lower bound for a fixed row and vector]			\label{lem: fixed row and vector}
  Let $G_j$ denote the $j$-th row of $G$.
  Then for each $j$, $z \in S_{\C}^{n-1}$, and $\theta \ge 0$, we have
  \begin{equation}							\label{eq: fixed row vector}
  \Pr{ |\ip{G_j}{z}| \le \theta} \le C_0 K \theta.		
  \end{equation}
\end{lemma}

\begin{proof}
Fix $j$ and consider the random vector $Z = G_j$. Expressing $Z$ and $z$
in terms of their real and imaginary parts as
$$
Z = X+iY, \quad z = x + iy,
$$
we can write the inner product as
$$
\ip{Z}{z} = \left[ \ip{X}{x} - \ip{Y}{y} \right] + i \left[ \ip{X}{y} + \ip{Y}{x} \right].
$$

Since $z$ is a unit vector, either $x$ or $y$ has norm at least $1/2$.
Assume without loss of generality that $\|x\|_2 \ge 1/2$. Dropping the imaginary part, we obtain
$$
|\ip{Z}{z}| \ge \left| \ip{X}{x} - \ip{Y}{y} \right|.
$$
Recall that the imaginary part $Y$ is fixed by Assumption~\ref{A}. Thus
\begin{equation}         \label{eq: SBP Zz}
\Pr{ |\ip{Z}{z}| \le \theta} \le \LL( \ip{X}{x}, \theta ).
\end{equation}

We can express $\ip{X}{x}$ in terms of the coordinates of $X$ and $x$ as the sum
$$
\ip{X}{x} = \sum_{k=1}^n X_k x_k.
$$
Since $X_k$ are the real parts of independent entries of $G$,
Assumptions~\ref{A} and \ref{A: continuous distribution} imply that
$X_k$ are independent random variables with densities bounded by $K$.
Recalling that $\sum_{k=1}^n x_k^2 \ge 1/2$,
we can apply a known result about the densities of sums of independent random variables, see \cite{RV small ball}.
It states that the density of $\sum_{k=1}^n X_k x_k$
is bounded by $C K$. It follows that
\begin{equation}         \label{eq: SBP Xx}
\LL( \ip{X}{x}, \theta ) \le C K \theta.
\end{equation}
Substituting this into \eqref{eq: SBP Zz} completes the proof of Lemma~\ref{lem: fixed row and vector}.
\end{proof}

\begin{lemma}[Lower bound for a fixed vector]		\label{lem: fixed vector}
  For each $x \in S_{\C}^{n-1}$ and $\theta >0$, we have
  $$
  \Pr{ \|Gx\|_2 \le \theta \sqrt{\e n} } \le (C_0 K \theta)^{\e n}.
  $$
\end{lemma}

\begin{proof}
We can represent $\|Gx\|_2^2$ as a sum of independent random variables
$\sum_{j=1}^{\e n} |\ip{G_j}{x} |^2$.
Each of the terms $\ip{G_j}{x}$ satisfies \eqref{eq: fixed row vector}.
Then the conclusion follows from Tensorization Lemma~\ref{lem: tensorization}.
\end{proof}

\subsubsection{Lower bound on a subspace}

\begin{lemma}[Lower bound on a subspace]		\label{lem: G bounded below on E}
  Let $M \ge 1$ and $\mu \in (0,1)$.
  Let $E$ be a fixed subspace of $\C^n$ of dimension at most $\mu \e n$.
  Then, for every $\theta >0$, we have
  \begin{equation}							\label{eq: G bounded below on E}
  \Pr{\inf_{x \in S_E} \|Gx\|_2 < \theta \sqrt{\e n} \text{ and } \BB_{G,M}}
  \le \left[ C K (M/\sqrt{\e})^{2 \mu} \theta^{1-2 \mu} \right]^{\e n}.
  \end{equation}
\end{lemma}

\begin{proof}
Let $\d \in (0,1)$ to be chosen later.
Since the dimension of $\real{E} \subset \R^{2n}$ is at most $2\mu \e n$, standard volume considerations imply the existence of a $\delta$-net $\NN \subset S_E$ with
\begin{equation}							\label{eq: NN size}
|\NN| \le \Big(\frac{3}{\delta} \Big)^{2 \mu \e n}.
\end{equation}
Assume that the event in the left hand side of \eqref{eq: G bounded below on E} occurs.
Choose $x \in S_E$ such that $\|Gx\|_2 < \theta \sqrt{\e n}$. Next, choose $x_0 \in \NN$ such that
$\|x-x_0\|_2 \le \d$. By triangle inequality and using $\BB_{G,M}$, we obtain
$$
\|Gx_0\|_2 \le \|Gx\|_2 + \|G\| \cdot \|x-x_0\|_2
\le \theta \sqrt{\e n} + M \sqrt{n} \cdot \d
$$
Choosing $\d := \theta \sqrt{\e} / M$, we conclude that $\|Gx_0\|_2 \le 2 \theta \sqrt{\e n}$.

Summarizing, we obtained that the probability of the event in the left hand
side of \eqref{eq: G bounded below on E} is bounded by
$$
\Pr{ \exists x_0 \in \NN: \; \|Gx_0\|_2 \le 2 \theta \sqrt{\e n} }.
$$
By Lemma~\ref{lem: fixed vector} and a union bound, this in turn is bounded by
$$
|\NN| \cdot (2C_0K \theta)^{\e n} \le \Big(\frac{3M}{\theta \sqrt{\e}} \Big)^{2 \mu \e n} \cdot (2C_0K \theta)^{\e n},
$$
where we used \eqref{eq: NN size} and our choice of $\delta$.
Rearranging the terms completes the proof.
\end{proof}

\subsubsection{Conclusion: $G$ is bounded below on a large subspace $E^-$}		\label{s: G bdd on E-}

We can apply Lemma~\ref{lem: G bounded below on E}
for the subspace $E = E^-$ constructed in Section~\ref{s: E+}.
We do this conditionally on $B$, for a fixed choice of $E^-$ that satisfies $\DD_{E^-}$,
thus for $\mu \le c < 0.05$. This yields the following.

\begin{lemma}[$G$ is bounded below on $E^-$]			\label{lem: G bounded below}
  For every $\theta >0$, we have
  $$
  \Pr{\inf_{x \in S_{E^-}} \|Gx\|_2 < \theta \sqrt{\e n} \text{ and } \DD_{E^-} \text{ and } \BB_{G,M}}
  \le \big( C K M^{0.1} \e^{-0.05} \theta^{0.9} \big)^{\e n}.
  $$
  Recall that $\DD_{E^-}$ is the likely event defined in \eqref{eq: E- large}. \qed
\end{lemma}

\subsection{Proof of invertibility}	 \label{s: deduction of invertibility}

\subsubsection{Decomposing invertibility}

The following lemma reduces invertibility of $\bar{A}$ to invertibility of $B$ on $E^+$
and $G$ on $E^-$.

\begin{lemma}[Decomposition]			\label{lem: A decomposition}
  Let $A$ be an $m \times n$ matrix. Let us decompose $A$ as
  $$
  A =
  \begin{bmatrix}
  \phantom{X} B \phantom{X} \\ \phantom{X} G \phantom{X}
  \end{bmatrix},
  \quad B \in \C^{m_1 \times n}, \; G \in \C^{m_2 \times n}, \; m=m_1+m_2.
  $$
  Consider the orthogonal decomposition
  $\C^n = E^- + E^+$
  where $E^-$ and $E^+$ are eigenspaces\footnote{In other words, $E^-$ and $E^+$
  are the spans of two disjoint subsets of right singular vectors of $B$.}
   of $B^* B$. Denote
  $$
  s_A = \smin(A), \; s_B = \smin(B|_{E^+}), \; s_G = \smin(G|_{E^-}).
  $$
  \begin{equation}         \label{eq: A decomposition}
  s_A \ge \frac{s_B s_G}{4 \|A\|}.
  \end{equation}
\end{lemma}

\begin{proof}
Let $x \in S^{n-1}$. We consider the orthogonal decomposition
$$
x = x^- + x^+, \quad x^- \in E^-, \, x^+ \in E^+.
$$
We can also decompose $Ax$ as
$$
\|Ax\|_2^2 = \|Bx\|_2^2 + \|Gx\|_2^2.
$$
Let us fix a parameter $\l \in (0,1/2)$ and consider two cases.

\smallskip

{\em Case 1: $\|x^+\|_2 \ge \l$.} Then
$$
\|Ax\|_2 \ge \|Bx\|_2 \ge \|Bx^+\|_2 \ge s_B \cdot \l.
$$

{\em Case 2: $\|x^+\|_2 < \l$.}
In this case, $\|x^-\|_2 = \sqrt{1-\|x^+\|_2^2} \ge 1/2$.
Thus
\begin{align*}
\|Ax\|_2
  &\ge \|Gx\|_2 \ge \|Gx^-\|_2 - \|Gx^+\|_2 \\
  &\ge \|Gx^-\|_2 - \|G\| \cdot \|x^+\|_2
  \ge s_G \cdot \frac{1}{2} - \|G\| \cdot \l.
\end{align*}
Using that $\|G\| \le \|A\|$, we conclude that
$$
s_A = \inf_{x \in S^{n-1}} \|Ax\|_2 \ge \min \Big( s_B \cdot \l, \; s_G \cdot \frac{1}{2} - \|A\| \cdot \l \Big).
$$
Optimizing the parameter $\l$, we conclude that
$$
s_A \ge \frac{s_B s_G}{2 (s_B + \|A\|)}.
$$
Using that $s_B$ bounded by $\|A\|$, we complete the proof.
\end{proof}

\subsubsection{Proof of the Invertibility Theorem~\ref{thm: invertibility continuous}} \label{s: proof of invertibility continuous}

We apply Lemma~\ref{lem: A decomposition} for the matrix $\bar{A}$ and the
decomposition \eqref{eq: A decomposed} and obtain
$$
s_B s_G \le 4 \|\bar{A}\| s_{\bar{A}}.
$$
Since $\bar{A}$ is a sub-matrix of $A$, we have $\|\bar{A}\| \le M \sqrt{n}$ on the event $\BB_{A,M}$.
Further, \eqref{eq: B bounded below} yields the bound $s_B \ge c \tau \e \sqrt{n}$.
It follows that
\begin{align}
\Pr{s_{\bar{A}} < t \sqrt{n} \text{ and } \BB_{A,M}}
&\le \Pr{s_G < \frac{4Mt}{c \tau \e} \cdot \sqrt{n} \text{ and } \BB_{A,M}} \nonumber\\
&\le \Pr{s_G < \frac{4Mt}{c \tau \e^{3/2}} \cdot \sqrt{\e n} \text{ and } \DD_{E^-} \text{ and }\BB_{A,M}}
  + \P((\DD_{E^-})^c).	\label{eq: sA bar}
\end{align}
The last line prepared us for an application of Lemma~\ref{lem: G bounded below}.
Using this lemma along with the trivial inclusion $\BB_{A,M} \subseteq \BB_{G,M}$ and
the estimate \eqref{eq: E- large} of the probability of $(\DD_{E^-})^c$,
we bound the quantity in \eqref{eq: sA bar} by
$$
\left[ C K M^{0.1} \e^{-0.05} \Big( \frac{4Mt}{c \tau \e^{3/2}} \Big)^{0.9} \right]^{\e n} + (C_2 K \tau)^{\e n}.
$$
This bound holds for all $\tau, t > 0$. Choosing $\tau = \sqrt{t}$ and rearranging the terms,
we obtain
$$
\Pr{s_{\bar{A}} < t \sqrt{n} \text{ and } \BB_{A,M}}
\le \left[C K M \e^{-1.4} t^{0.45} \right]^{\e n} + (C K t^{0.5})^{\e n}.
$$
This implies the desired conclusion \eqref{eq: invertibility continuous simplified}.
Theorem~\ref{thm: invertibility continuous} is proved.
\qed

\section{Invertibility for general distributions: statement of the result}	\label{s: invertibility general}

We are now passing to random matrices whose entries may have general, possibly discrete distributions;
our goal being Delocalization Theorem~\ref{thm: delocalization general}.
Recall that in Section~\ref{s: reduction intro} we described on the informal level
how delocalization can be reduced to invertibility of random matrices;
Proposition~\ref{prop: reduction} formalizes this reduction.
This prepares us to state an invertibility result for general random matrices,
whose proof will occupy the rest of this paper.

\begin{theorem}[Invertibility: general distributions]			\label{thm: invertibility general}
  Let $A$ be an $n \times n$ random matrix satisfying the assumptions of
  Theorem~\ref{thm: delocalization general}.
  Let $M \ge 1$, $\e \in (1/n,c)$, and let $I \subset [n]$ be any fixed subset with $|I| = \e n$.
  Then for any
  \begin{equation}   \label{eq: t large}
    t \ge \frac{c}{\e n} + e^{-c/\sqrt{\e}},
  \end{equation}
   we have
  $$
  \Pr{\smin(A_{I^c}) \le t \sqrt{n} \text{ and } \BB_{A,M}}
  \le \left( Ct^{0.4} \e^{-1.4} \right)^{\e n/2}.
  $$
\end{theorem}

The constant $C$ in the inequality above depends on $M$ and the parameters $p$ and $K$ appearing in Assumption~\ref{A: general distribution}.

Delocalization Theorem~\ref{thm: delocalization general} follows from Theorem~\ref{thm: invertibility general} along the same lines as in Section~\ref{s: delocalization continuous}. As in that section, for a given $s$ we set $t=8M(\e s)^6$, which leads to the particular form of the restriction on $s$ in Theorem~\ref{thm: delocalization general}.
 This restriction, as well as the probability estimate, can be improved by tweaking various parameters throughout the proof of Theorem~\ref{thm: invertibility general}.
 They can be further and more significantly improved by taking into account the arithmetic structure in the small ball probability estimate, instead of disregarding it in Section \ref{s: SBP via correlations}.
We refrained from pursuing these improvements in order to avoid overburdening the paper with technical calculations.

\section{Small ball probabilities via least common denominator}     \label{s: LCD and SBP}

In this section, which may have an independent interest, we relate the sums of independent
random variables and random vectors to the arithmetic structure of their coefficients.

To see the relevance of this topic to invertibility of random matrices, we could
try to extend the argument we gave Section~\ref{s: continuous} to general distributions. Most of the argument
would go through. However, a major difficulty occurs when we try to estimate the distance
between a random vector $X$ and a fixed subspace $H$. For discrete distributions, $\dist(X,H) = \|P_{H^\perp} X\|_2$
can no longer be bounded below as easily as we did in Lemma~\ref{lem: distance}.
The source of difficulty can be best seen if we consider the simple example where $H$ is the
hyperplane orthogonal to the vector $(1,1,0,0,\ldots,0)$ and $X$ is the random Bernoulli vector
(whose coefficients are independent and take values $1$ and $-1$ with probability each).
In this case, $\dist(X,H)$ with probability $1/2$. Even if we exclude zeros by making
$H$ orthogonal to $(1,1,1,1,\ldots,1)$, the distance would equal zero with probability $\sim 1/\sqrt{n}$,
thus polynomially rather than exponentially fast in $n$.

The problem with these examples is that $H^\perp$ had rigid arithmetic structure.
In Section~\ref{s: LCD}, we will show how to quantify arithmetic structure with a notion of
approximate least common denominator (LCD). In Section~\ref{s: SBP via LCD}
will also provide bounds on sums of independent random vectors in terms of LCD.
Finally, in Sections~\ref{s: special cases sums} and \ref{s: special cases projections}
we will specialize these bounds for
sums of independent random variables and projections of random vectors
(and in particular, for distances to subspaces).

\subsection{The least common denominator}		\label{s: LCD}

An approximate concept of least common denominator (LCD) was proposed in  \cite{RV square}
to quantify the arithmetic structure of vectors; this idea was developed in
\cite{RV rectangular, RV upper, V symmetric}, see also \cite{RV ICM}. Here we will use the
version of LCD from \cite{V symmetric}.
We emphasize that throughout this  section we consider \emph{real} vectors and matrices.

\begin{definition}[Least common denominator]			\label{def: LCD}
  Fix $L > 0$. For a {\em vector} $v \in \R^N$, the least common denominator (LCD) is defined as
  $$
  D(v) = D(v,L) = \inf \left\{ \theta>0: \; \dist(\theta v, \Z^N) < L \sqrt{\log_+ \frac{\|\theta v\|_2}{L}} \right\}.
  $$
  For a {\em matrix} $V \in \R^{m \times N}$, the least common denominator is defined as
  $$
  D(V)= D(V,L) = \inf \left\{ \|\theta\|_2: \; \theta \in \R^m, \;
    \dist(V^\tran \theta, \Z^N) < L \sqrt{\log_+ \frac{\|V^\tran \theta\|_2}{L}} \right\}.
  $$
\end{definition}

\begin{remark}
  The definition of LCD for vectors is a special case of the definition for matrices with $m=1$.
  This can be seen by considering a vector $v \in \R^N$ as a $1 \times N$ matrix.
\end{remark}

\begin{remark}
  In applications, we will typically choose $L \sim \sqrt{m}$, so for vectors we usually choose $L \sim 1$.
\end{remark}

Before relating the concept LCD to small ball probabilities, let us pause to note a simple but useful
lower bound for LCD.
To state it, for a given matrix $V$ we let $\|V\|_\infty$ denote the maximum Euclidean norm
of the columns of $V$.
Note that for vectors ($1 \times N$ matrices), this quantity is the usual $\ell_\infty$ norm.

\begin{proposition}[Simple lower bound for LCD]		\label{prop: LCD lower simple}
  For every matrix $V$ and $L > 0$, one has
  $$
  D(V,L) \ge \frac{1}{2\|V\|_\infty}.
  $$
\end{proposition}	

\begin{proof}
By definition of LCD, it is enough to show that for $\theta \in \R^m$, the inequality
\begin{equation}         \label{eq: theta as in LCD}
\dist(V^\tran \theta, \Z^N) < L \sqrt{\log_+ \frac{\|V^\tran \theta\|_2}{L}}
\end{equation}
implies $\|\theta\|_2 \ge 1/(2\|V\|_\infty)$. Assume the contrary, that there exists $\theta$
which satisfies \eqref{eq: theta as in LCD} but for which
\begin{equation}         \label{eq: theta too small}
\|\theta\|_2 < \frac{1}{2\|V\|_\infty}.
\end{equation}
We can use Cauchy-Schwartz inequality and \eqref{eq: theta too small}
to bound all coordinates $\ip{V_j}{\theta}$ of the vector $V^\tran \theta$ as follows:
$$
|\ip{V_j}{\theta}| \le \|V\|_\infty \|\theta\|_2 < \frac{1}{2}, \quad j =1,\ldots,N.
$$
(Here $V_j \in \R^m$ denote the columns of the matrix $V$.)
This bound means that each coordinate of $V^\tran \theta$ is closer to zero
than to any other integer. Thus the vector $V^\tran \theta$ itself is closer (in the $\ell_2$ norm)
to the origin than to any other integer vector in $\Z^N$. This implies that
$$
\dist(V^\tran \theta, \Z^N) = \|V^\tran \theta\|_2.
$$
Substituting this into \eqref{eq: theta as in LCD} and dividing both sides by $L$,
we obtain
$$
u \le \sqrt{\log_+ u} \quad \text{where} \quad u = \|V^\tran \theta\|_2/L.
$$
But this inequality has no solutions for $u>0$.
This contradiction completes the proof.
\end{proof}

\subsection{Small ball probabilities via LCD}			\label{s: SBP via LCD}

The following theorem relates small ball probabilities to arithmetic structure, which is
measured by LCD. It is a general version of results from \cite{RV square, RV rectangular, V symmetric}.

\begin{theorem}[Small ball probabilities via LCD]			\label{thm: SBP LCD}
  Consider a random vector $\xi = (\xi_1,\ldots,\xi_N)$, where $\xi_k$ are i.i.d. copies
  of a real-valued random variable $\xi$ satisfying \eqref{eq: xi}.
  Consider a matrix $V \in \R^{m \times N}$. Then for every $L \ge \sqrt{8m/p}$
  we have
  \begin{equation}							\label{eq: SBP LCD}
  \LL(V\xi, t \sqrt{m}) \le \frac{(CL/\sqrt{m})^m}{\det(VV^\tran)^{1/2}}
    \Big( t + \frac{\sqrt{m}}{D(V)} \Big)^m, \quad t \ge 0.
  \end{equation}
\end{theorem}

\begin{proof}
We shall apply Esseen's inequality for the small ball probabilities of a general
random vector $Y \in \R^m$. It states that
\begin{equation}							\label{eq: Esseen}
\LL(Y, \sqrt{m}) \le C^m \int_{B(0,\sqrt{m})} |\phi_Y(\theta)| \, d\theta
\end{equation}
where $\phi_Y(\theta) = \E \exp(2 \pi i \ip{\theta}{Y})$ is the characteristic function of $Y$
and $B(0,\sqrt{m})$ is the Euclidean ball centered at the origin and with radius $\sqrt{m}$.

Let us apply Esseen's inequality for $Y = t^{-1}V\xi$, assuming without loss of generality
that $t>0$.
Denoting the columns of $V$ by $V_k$, we express
$$
\ip{\theta}{Y} = \sum_{k=1}^N t^{-1} \ip{\theta}{V_k} \xi_k.
$$
By independence of $\xi$, this yields
$$
\phi_Y(\theta) = \prod_{k=1}^N \phi_k(t^{-1}\ip{\theta}{V_k}),
\quad \text{where} \quad
\phi_k(\tau) = \E \exp(2 \pi i \tau \xi_k)
$$
are the characteristic functions of $\xi_k$.
Therefore, Esseen's inequality \eqref{eq: Esseen} yields
\begin{equation}							\label{eq: LL via phik}
\LL(V\xi, t \sqrt{m}) = \LL(Y, \sqrt{m})
\le C^m \int_{B(0,\sqrt{m})} \prod_{k=1}^N \big| \phi_k(t^{-1}\ip{\theta}{V_k}) \big| \, d\theta.
\end{equation}

Now we evaluate the characteristic functions that appear in this integral.
First we apply a standard symmetrization argument. Let $\xi'$ be an independent
copy of $\xi$, and consider the random vector $\bar{\xi} := \xi - \xi'$.
Its coordinates $\bar{\xi}_k$ are i.i.d. random variables with symmetric distribution.
It follows that
$$
|\phi_k(\tau)|^2 = \E \exp(2 \pi i \tau \bar{\xi}_1) = \E \cos(2 \pi \tau \bar{\xi}_1) \quad \text{for all $k$}.
$$
Using the inequality $x \le \exp(-\frac{1}{2}(1-x^2))$
that is valid for all $x \ge 0$, we obtain
\begin{equation}							\label{eq: phik}
|\phi_k(\tau)| \le \exp \Big[ -\frac{1}{2} \E \big[ 1 - \cos(2 \pi \tau \bar{\xi}_1) \big] \Big].
\end{equation}
The assumptions on $\xi_k$ imply that the event $\{ 1 \le |\bar{\xi}_1| \le K\}$
holds with probability at least $p/2$.
Denoting by $\bar{\E}$ the conditional expectation on that event, we obtain
\begin{equation}         \label{eq: E 1-cos}
\E \big[ 1 - \cos(2 \pi \tau \bar{\xi}_1) \big]
\ge \frac{p}{2} \bar{\E} \big[ 1 - \cos(2 \pi \tau \bar{\xi}_1) \big]
\ge \frac{p}{2} \cdot \bar{\E} \min_{q \in \Z} |\tau \bar{\xi}_1 - q|^2.
\end{equation}

Substituting this inequality into \eqref{eq: phik} and then back into \eqref{eq: LL via phik},
we further derive
\begin{align}					\label{eq: integral of f}
\LL(V\xi, t \sqrt{m})
&\le C^m \int_{B(0,\sqrt{m})} \exp \Big[ - \frac{p}{4} \bar{\E} \sum_{k=1}^N
  \min_{q_k \in \Z} |t^{-1} \bar{\xi}_1 \ip{\theta}{V_k} - q_k|^2 \Big] \, d\theta \nonumber\\
&= C^m \int_{B(0,\sqrt{m})} \exp \big( -\frac{p}{4} f(\theta)^2 \big) \, d\theta
\end{align}
where
$$
f(\theta)^2 := \bar{\E} \min_{q \in \Z^N} \Big\| t^{-1} \bar{\xi}_1 V^\tran \theta - q \Big\|_2^2
= \bar{\E} \dist( t^{-1} \bar{\xi}_1 V^\tran \theta, \Z^N )^2.
$$

The least common denominator $D(V,L)$ will help us
estimate the distance to the integer lattice that appears in the definition of $f(\theta)$.
Let us first assume that
\begin{equation}							\label{eq: epsilon large}
t \ge t_0 := \frac{2 K \sqrt{m}}{D(V,L)},
\end{equation}
or equivalently that $D(V,L) \ge 2 K \sqrt{m}/t$.
Then for any $\theta$ appearing in the integral \eqref{eq: integral of f},
that is for $\theta \in B(0,\sqrt{m})$, one has
$$
\| t^{-1} \bar{\xi}_1 \theta \|_2 \le K t^{-1} \sqrt{m} < D(V).
$$
(Here we used that $|\bar{\xi}_1| \le K$ holds on the event over which the conditional
expectation $\bar{\E}$ is taken.)
By the definition of $D(V)$, this implies that
$$
\dist(V^\tran(t^{-1} \bar{\xi}_1 \theta), \Z^N)
\ge L \sqrt{\log_+ \frac{\|V^\tran(t^{-1} \bar{\xi}_1 \theta)\|_2}{L}}.
$$
Recalling the definition of $f$ and using that $|\bar{\xi}_1| \ge 1$ on the event over which the conditional
expectation $\bar{\E}$ is taken, we obtain
$$
f(\theta)^2 \ge L^2 \log_+ \frac{\|V^\tran \theta\|_2}{Lt}.
$$
where in the second inequality we use that $|\bar{\xi}_1| \ge 1$ on the event over
which the conditional expectation $\bar{\E}$ is taken.
Substituting this bound into \eqref{eq: integral of f}, we obtain
$$
\LL(V\xi, t \sqrt{m})
\le C^m \int_{B(0,\sqrt{m})} \exp \Big( -\frac{pL^2}{4} \log_+ \frac{\|V^\tran \theta\|_2}{Lt} \Big) \, d\theta.
$$
One can estimate this integral in a standard way.

Let us get rid of $V$ in the integrand by an appropriate change of variable.
Using a singular value decomposition of $V$, one can replace $V^\tran \theta$ by $\Sigma \theta$
where $\Sigma \in \R^{m \times m}$ is a diagonal matrix with singular values of $V$ on the diagonal.
Next, we change variables to $\Sigma \theta / Lt = z$. Since $\det \Sigma = \det(V V^\tran)^{1/2}$,
this yields
\begin{equation}         \label{eq: SBP via exponential integral}
\LL(V\xi, t \sqrt{m}) \le \frac{(CLt)^m}{\det(VV^\tran)^{1/2}}
\int_{\R^m} \exp \Big( -\frac{pL^2}{4} \log_+ \|z\|_2 \Big) \, dz.
\end{equation}
We evaluate the integral by breaking it into two parts:
\begin{equation}         \label{eq: integral broken}
\int_{\R^m} \exp \Big( -\frac{pL^2}{4} \log_+ \|z\|_2 \Big) \, dz
= \int_{B(0,1)} 1 \, dz + \int_{B(0,1)^c} \|z\|_2^{-pL^2/4} \, dz.
\end{equation}
Let us start with the second integral. Passing to the polar coordinates
$(r,\phi) \in \R_+ \times S^{m-1}$ where $dz = r^{m-1} \, dr \, d\phi$,
we obtain for any $q \ge 0$ that
$$
\int_{B(0,1)^c} \|z\|_2^{-q} \, dz
= \int_1^\infty dr \int_{S^{m-1}} r^{-q} r^{m-1} \, d\phi
= \s_{m-1}(S^{m-1}) \int_1^\infty r^{m-q-1} \, dr,
$$
where $\s_{m-1}(S^{m-1})$ is the surface area of the unit sphere.
Recall that
$$
\s_{m-1}(S^{m-1}) = \frac{2 \pi^{m/2}}{\Gamma(m/2)} \le \Big(\frac{C}{\sqrt{m}}\Big)^m
$$
and that
$$
\int_1^\infty r^{m-q-1} \, dr = \frac{1}{q-m} \le 1
\quad \text{for} \quad q \ge 2m.
$$
This yields
$$
\int_{B(0,1)^c} \|z\|_2^{-q} \, dz \le \Big(\frac{C}{\sqrt{m}}\Big)^m
\quad \text{for} \quad q \ge 2m.
$$
We use this bound for $q = pL^2/4$, where $q \ge 2m$ by assumption.
It follows that the integral over $B(0,1)^c$ in \eqref{eq: integral broken}
is bounded by $(C/\sqrt{m})^m$. Moreover, the integral over $B(0,1)$
equals the volume of the unit ball $B(0,1)$, which is also bounded by $(C/\sqrt{m})^m$.
Thus the right hand side of \eqref{eq: integral broken} is bounded by $2(C/\sqrt{m})^m$.
Substituting it into \eqref{eq: SBP via exponential integral}, we obtain
\begin{equation}         \label{eq: SBP for large t}
\LL(V\xi, t \sqrt{m}) \le \frac{2(CLt)^m}{\det(VV^\tran)^{1/2}} \Big(\frac{C}{\sqrt{m}}\Big)^m.
\end{equation}
This completes the proof in the case where $t \ge t_0$ as specified in \eqref{eq: epsilon large}.

In the opposite case where $t \le t_0$, it is enough to use that
$\LL(V\xi, t \sqrt{m}) \le \LL(V\xi, t_0 \sqrt{m})$ and apply the inequality \eqref{eq: SBP for large t}
for $t_0$. This completes the proof of Theorem~\ref{thm: SBP LCD}.
\end{proof}

\subsection{Special cases: sums of independent random variables}				\label{s: special cases sums}

Let us state an immediate consequence of Theorem~\ref{thm: SBP LCD} in the important special case
where $m=1$. In this case, $V\xi$ becomes a sum of independent random variables.

\begin{corollary}[Small ball probabilities for sums]		\label{cor: SBP sums}
  Let $\xi_k$ be i.i.d. copies of a real-valued random variable $\xi$ satisfying \eqref{eq: xi}.
  Let $a = (a_1,\ldots, a_N) \in \R^n$.
  Then for every $L \ge \sqrt{8/p}$ we have
  $$
  \LL \Big( \sum_{k=1}^N a_k \xi_k, t \Big)
  \le \frac{CL}{\|a\|_2} \Big( t + \frac{1}{D(a,L)} \Big), \quad t \ge 0.
  $$
\end{corollary}

This corollary was proved in \cite{V symmetric}; similar versions appeared in \cite{RV square, RV rectangular}.
It gives a non-trivial probability bound when the coefficient vector $a$ is sufficiently unstructured, i.e. when
$D(a,L)$ is large enough. In the situation where no information is known about the structure of $a$,
the following result can be useful.

\begin{lemma}[Small ball probabilities: a simple bound]			\label{lem: SBP simple}
  Let $\xi_k$ be independent random variables satisfying \eqref{eq: xi},
  and let $a_j$ be real numbers such that $\sum_{j=1}^N a_j^2 = 1$.
  Then
  $$
  \LL \Big( \sum_{k=1}^N a_k \xi_k, c \Big) \le 1-c'
  $$
  where $c$ and $c'$ are positive numbers that may only depend on $p$ and $K$.
\end{lemma}

\begin{proof}
We will consider separately the cases where $a$ has a large coordinate and where it does not.
Assume first that
$$
\|a\|_{\infty} \ge \frac{1}{4CL} =: \nu
$$
where $C$ is the constant appearing in Corollary \ref{cor: SBP sums}.
Choose a coordinate $k_0$ such that $|a_{k_0}|= \|a\|_{\infty}$.
Applying Lemma~\ref{lem: SBP restriction}, we obtain
 \[
  \LL \Big( \sum_{k=1}^N a_k \xi_k, \nu \Big)
  \le \LL \Big(  a_{k_0} \xi_{k_0}, \nu \Big)
  \le \LL( \xi_{k_0}, 1 )
  \le 1-p.
 \]

In the opposite case where $\|a\|_{\infty} < \nu$, Proposition~\ref{prop: LCD lower simple} implies
$D(a,L) \ge 1/(2 \nu)$. Combining this with Corollary~\ref{cor: SBP sums}, we obtain
 \[
   \LL \Big( \sum_{j=1}^N a_j \xi_j, \nu \Big)
   \le CL \cdot 3 \nu \le \frac{3}{4},
 \]
 which completes the proof.
\end{proof}

\subsection{Special cases: projections of random vectors}		\label{s: special cases projections}

Another class of examples where Theorem~\ref{thm: SBP LCD} is useful is
for projections of a random vector $\xi$ onto a fixed subspace $E$ of $\R^N$.
Equivalently, this result allows us to estimate the distances between random vectors and
fixed subspaces, since $\dist(X,H) = \|P_{H^\perp} X\|_2$.

To deduce such estimates, we will make the matrix $V$ Theorem~\ref{thm: SBP LCD} encode
an orthogonal projection onto $E$. Let us pause to interpret the LCD of such matrix $V$ as the LCD of the subspace $E$ itself.

\begin{definition}[LCD of a subspace]			\label{def: LCD subspace}
  Fix $L>0$.
  For a subspace $E \subseteq \R^N$, the least common denominator is defined as
  $$
  D(E) = D(E, L) = \inf \{ D(v,L): \; v \in S_E\}.
  $$
\end{definition}

By now, we have defined LCD of vectors, matrices, and subspaces. The following lemma
relates them together.

\begin{lemma}[LCD of subspaces vs. matrices] 			\label{lem: LCD subspaces matrices}
  Let $E$ be a subspace of $\R^N$. Then
  \begin{enumerate}
    \item $D(E) = \inf \left\{ \|x\|_2 :\; x \in E, \; \dist(x, \Z^N) < L \sqrt{\log_+ \frac{\|x\|_2}{L}} \right\}$.
    \item Let $U \in \R^{N \times m}$ be a matrix such that $U^\tran U = I_m$ and $\im(U) = E$.
    Then $D(E) = D(U^\tran)$.
  \end{enumerate}
\end{lemma}

\begin{proof}
The first part follows directly from the definition.
To prove the second part, note that according to Definition~\ref{def: LCD} we have
$$
D(U^\tran) = \inf \left\{ \|\theta\|_2: \; \theta \in \R^m, \;
  \dist(U \theta, \Z^N) < L \sqrt{\log_+ \frac{\|U \theta\|_2}{L}} \right\}.
$$
Let us change variable to $x = U \theta$.
The assumptions on $U$ imply that $\|x\|_2 = \|\theta\|_2$ and
as $\theta$ runs over $\R^m$, $x$ runs over $\im(U) = E$.
We finish by applying the first part of this lemma.
\end{proof}

The following corollary is a version of a result from \cite{RV rectangular}.

\begin{corollary}[Small ball probabilities for projections]					\label{cor: SBP proj}
  Consider a random vector $\xi = (\xi_1,\ldots,\xi_N)$, where $\xi_k$ are i.i.d. copies
  of a real-valued random variable $\xi$ satisfying \eqref{eq: xi}.
  Let $E$ be a subspace of $\R^N$ with $\dim(E) = m$, and
  let $P_E$ denote the orthogonal projection onto $E$.
  Then for every $L \ge \sqrt{8m/p}$ we have
  \begin{equation}							\label{eq: SBP proj}
  \LL(P_E \xi, t \sqrt{m}) \le \Big( \frac{CL}{\sqrt{m}} \Big)^m
    \Big( t + \frac{\sqrt{m}}{D(E,L)} \Big)^m, \quad t \ge 0.
  \end{equation}
\end{corollary}

\begin{proof}
Choose a matrix $U \in \R^{N \times m}$ so that $U^\tran U = I_m$ and $UU^\tran = P_E$.
Then $U$ acts as an isometric embedding from $\R^m$ into $\R^N$, i.e. $\|Ux\|_2 = \|x\|_2$
for all $x \in \R^m$. This yields
$$
\LL(P_E \xi, t \sqrt{m}) = \LL(UU^\tran \xi, t \sqrt{m}) = \LL(U^\tran \xi, t \sqrt{m}).
$$
We apply Theorem~\ref{thm: SBP LCD} for $V = U^\tran$ and note that
$\det(VV^\tran) = \det(U^\tran U) = \det(I_m) = 1$. Thus $\LL(P_E \xi, t \sqrt{m})$
gets bounded by the same quantity as in the right hand side of \eqref{eq: SBP proj}
except for $D(V)$. It remains to use Lemma~\ref{lem: LCD subspaces matrices},
which yields $D(V) = D(U^\tran) = D(E)$.
\end{proof}

\section{Distances between random vectors and subspaces: statement of the result}    \label{s: distance statement}
Our next goal is to prove a lower bound for the distance between independent random vectors
and subspaces. For continuous distributions, this was achieved in Lemma~\ref{lem: distance}.
Doing this for general, possibly discrete, distributions, is considerably more difficult.
The following result is a version of Lemma~\ref{lem: distance} for general distributions.

\begin{theorem}[Distance between random vectors and subspaces]				\label{thm: distance general}
  Let $H \in \C^{N \times n}$ be a random matrix which satisfies Assumptions~\ref{A}\footnote{Assumption \ref{A} is formulated for square random matrices. For rectangular matrices, one of the entries $A_{ij}$ or $A_{ji}$ may not exist. In this case, we assume that the other enty is independent of the rest.}
   and \ref{A: general distribution},
  and assume that $n = (1-\e) N$ for some $\e \in (2/n,c)$.
  Let $Z \in \C^N$ be a random vector independent of $H$, and whose coordinates
  are i.i.d. random variables satisfying the same distributional assumptions as specified in Assumption~\ref{A: general distribution}. Then
  $$
  \Pr{\dist(Z, \im(H)) \le \tau \sqrt{\e N} \  \textrm{and}  \ \BB_{H,M}}
  \le \left[ C \Big( \tau + \frac{1}{\sqrt{\e N}} + e^{-c/\sqrt{\e}} \Big) \right]^{\e N},
  \quad \tau \ge 0.
  $$
\end{theorem}

A version of this theorem was proved in \cite{RV rectangular} in the simpler situation where
the entries of $H$ are real-valued and all independent. In this simpler case, \cite{RV rectangular}
gives the following optimal bound:
$$
\Pr{\dist(Z, \im(H)) \le \tau \sqrt{\e N}} \le (C\tau)^{\e N} + e^{-cN}.
$$
We do not know if the same bound can be proved in the setting of Theorem~\ref{thm: distance general}.

\medskip

To prove Theorem~\ref{thm: distance general}, we will first reduce it to a problem over reals --
much like we did in Section~\ref{s: from C to R}.
Then, expressing the distance $\dist(Z, \im(H))$ as the norm of the projection of $Z$ onto
$\im(H)^\perp = \ker(H^*)$,
we should be able to apply Corollary~\ref{cor: SBP proj}.
However, for the resulting probability bound \eqref{eq: SBP proj}
to be meaningful, we would need to show that the least common denominator $D(\ker(H^*),L)$ is large,
or in other words, that the subspace $H$ is unstructured.
This will be a major step in the argument. Eventually we will achieve this in Section~\ref{s: distance proof},
which will allow us to quickly finalize the proof of Theorem~\ref{thm: distance general}.

\medskip

In preparation for the proof of Theorem~\ref{thm: distance general},
let us express the distance we need to estimate as follows:
\begin{equation}							\label{eq: dist as proj general}
\dist(Z, \im(H)) = \|P_{\im(H)^\perp} Z\|_2 = \|P_{\ker(B)} Z\|_2, \quad \text{where } B = H^* \in \C^{n \times N}.
\end{equation}
Our goal is to show that $\ker(B)$ is arithmetically unstructured.

\subsection{Transferring the problem from $\C$ to $\R$}			\label{s: from C to R general}

Similarly to our argument for continuous distributions, we will now transfer
the distance problem from the complex to the real field. In Section~\ref{s: from C to R}, we
introduced the operation $z \mapsto \real{z}$ that makes a complex vector $z = x+iy$ in $\C^N$ real
by defining $\real{z} := \vectwo{x}{y} \in \R^{2N}$.
We also introduced this operation for subspaces $E$ of $\C^N$ by defining
$\real{E} = \{ \real{z} :\; z \in E \} \subset \R^{2N}$.

In the analysis of the distance problem for continuous distributions,
we did not need to know anything about the subspaces $H'_j$ beyond their dimensions.
This time, our analysis will be sensitive to the structure of the subspace $\ker(B)$.
For this purpose, we will need to transfer the matrix $B$ from complex to real field.
We can do this in a way that preserves matrix-vector multiplication as follows:
\begin{equation}         \label{eq: real B}
\text{For } B = R + i T \in \C^{n \times N}, \text{ define }
\real{B} =
  \begin{bmatrix}
  R & -T \\
  T & R
  \end{bmatrix} \in \R^{2n \times 2N}.
\end{equation}

We already observed two elementary properties of the operation $z \mapsto \real{z}$
in Lemma~\ref{lem: real}; let us record one more straightforward fact.

\begin{lemma}[Elementary property of operation $x \mapsto \real{x}$]				\label{lem: real new}
  For a complex matrix $B$ and a vector $z$, one has $\real{Bz} = \real{B}{\real{z}}$,
  and consequently $\real{\ker(B)} = \ker(\real{B})$.
\end{lemma}

Let us return to the distance problem \eqref{eq: dist as proj general}.
Applying Lemma~\ref{lem: Z real} for $E = \ker B$, we conclude that
$$
\LL(P_{\ker(B)} Z, r) \le \LL(P_{\real{\ker(B)}} \widehat{Z}, 2r)^{1/2}.
$$
Using the interpretation of distance as norm of projection in \eqref{eq: dist as proj general},
we can summarize the first step toward the proof of Theorem~\ref{thm: distance general}.
We showed that
\begin{equation}         \label{eq: reduction of dist to R}
\Pr{\dist(Z, \im(H)) \le \tau \sqrt{\e N}} \le \LL(P_{\real{\ker(B)}} \widehat{Z}, 2 \tau \sqrt{\e N})^{1/2}.
\end{equation}

Recall that here, according to Lemma~\ref{lem: real new}, $\real{\ker B} = \ker \real{B}$,
where $\real{B}$ is the random matrix from \eqref{eq: real B}
and $\widehat{Z} \in \R^{2N}$ is a random vector.
Specifically, $T$ is a fixed $n \times N$ matrix
and $R$ is an $n \times N$ random matrix, which satisfies the structural and distributional
requirements of Assumptions~\ref{A} and \ref{A: general distribution} (except that $R$ entirely real).
The coordinates of the random vector $\widehat{Z}$ are i.i.d. copies of
a real random variable $\xi$ satisfying \eqref{eq: xi}.

\section{Kernels of random matrices are incompressible}		\label{s: kernels incompressible}

\subsection{Compressible and incompressible vectors}

Before we can show that the kernel of $B$ consists of arithmetically unstructured vectors,
we will prove a much simpler result. It states that the kernel of $B$ consists of
{\em incompressible vectors} -- those whose mass is not concentrated on a small
number of coordinates. The partition of the space into compressible and incompressible vectors
has been instrumental in arguments leading to invertibility random matrices,
see \cite{RV square, RV rectangular, V symmetric}.

\begin{definition}[Compressible and incompressible vectors] 		\label{def: Comp Incomp}
  Let $c_0, c_1 \in (0,1)$ be two constants.
  A vector $z \in \C^N$ is called {\em sparse} if $|\supp(z)| \le c_0 N$.
  A vector $z \in S_\C^{N-1}$ is called {\em compressible} if $x$
  is within Euclidean distance $c_1$ from the set of all sparse vectors.
  A vector $z \in S_C^{N-1}$ is called {\em incompressible} if it is not compressible.
  The sets of compressible and incompressible vectors in $S_\C^{N-1}$
  will be denoted by $\Comp$ and $\Incomp$ respectively.
\end{definition}

 The definition above depends on the choice of the constants $c_0,c_1$.
 These constants will be chosen in Proposition~\ref{prop: invertibility on Comp} and remain fixed throughout the paper.

As we already announced, our goal in this section is to prove that, with high probability, the kernel of $B$ consists entirely of incompressible vectors.
We will deduce this by providing a uniform lower bound for $\|Bz\|_2$ for all compressible vectors $z$.

\subsection{Relating $\|Bz\|_2$ to a sum of independent random variables}		\label{s: Bz as a sum}

Let us fix a vector $z$ for now.
We would like to rexpress $\|Bz\|_2$ to a sum of independent random variables, and then to use
bounds on small ball probabilities from Section~\ref{s: LCD and SBP}.
Using the real version $\real{B}$ of the matrix $B$, and the real version $\real{z}$ of the vector $z$
we introduced in Section~\ref{s: from C to R general}, we can write
\begin{equation}							\label{eq: Bz two terms}
\|Bz\|_2^2 = \|\real{B} \real{z}\|_2^2 = \|Rx+Ty\|_2^2 + \|Ry-Tx\|_2^2.
\end{equation}
Let us fix a subset $J \subset [N]$. Dropping the coefficients of the vectors $Rx+Ty$ and $Ry-Tx$
indexed by $J$, we obtain
$$
\|Bz\|_2^2 \ge \|R_{J^c \times [n]} x + a\|_2^2 + \|R_{J^c \times [n]} y - b\|_2^2,
$$
where $a = T_{J^c \times [n]} y$ and $b = T_{J^c \times [n]} x$ are fixed vectors.

Further, let us decompose $R_{J^c \times [n]} = R_{J^c \times J} + R_{J^c \times J^c}$, where $R_{J^c \times I}$ denotes the matrix $R_{J^c \times [n]}$ whose columns which does not belong to $I$ are replaced by zeros.
Assumption~\ref{A} implies that these two components are independent, and moreover
the first one, $R_{J^c \times J}$, has independent entries.
So let us condition on the second component, $R_{J^c \times J^c}$.
Absorbing its contribution into $a$ and $b$, we obtain
$$
\|Bz\|_2^2 \ge \|R_{J^c \times J} x + a'\|_2^2 + \|R_{J^c \times J} y - b'\|_2^2,
$$
where $a'$ and $b'$ are fixed vectors.
Expanding the matrix-vector multiplication, we arrive at the bound
\begin{equation}							\label{eq: Bz double sum restricted}
\|Bz\|_2^2
\ge \sum_{i \in [n] \setminus J} X_i^2 + Y_i^2,
\end{equation}
where
\begin{equation}         \label{eq: Xj Yj}
X_i = \sum_{j \in J} R_{ij} x_j + a'_j, \quad Y_i = \sum_{j \in J} R_{ij} y_j - b'_j
\end{equation}
and $a'_j$ and $b'_j$ are fixed numbers.
The sum in \eqref{eq: Bz double sum restricted} should be convenient to control,
since all $R_{ij}$ appearing in \eqref{eq: Xj Yj} are independent random variables.
\bigskip

\subsection{A lower bound on $\|Bz\|_2$ for compressible vectors}		\label{s: compressible}

We start with a simple and general lower bound on $\|Bz\|_2$ for a fixed vector $z$.

\begin{proposition}[Matrix acting on a fixed vector: simple bound]			\label{prop: Bz simple}
  Let $n \le N \le 2n$, and $B \in \C^{n \times N}$ be a random matrix satisfying Assumptions~\ref{A} and \ref{A: general distribution}.
  Then for any fixed vector $z \in \C^N$ with $\|z\|_2 = 1$ we have
  $$
  \Pr{ \|Bz\|_2 \le c\sqrt{n} }  \le e^{-cn}.
  $$
\end{proposition}

\begin{proof}
Let $z = x+iy$, and choose $J$ to be the set of indices of the $N/4$ largest coordinates of $z$.
Since $z$ is a unit vector, we have
$$
\|x_J\|_2^2 + \|y_J\|_2^2 = \|z_J\|_2^2 \ge \frac{1}{4}.
$$
It follows that either $x_J$ or $y_J$ has norm at least $1/4$.
Without loss of generality, let us assume that $\|x_J\|_2 \ge 1/4$.

Dropping the terms $Y_j$ from \eqref{eq: Bz double sum restricted}, we see that
\begin{equation}         \label{eq: Bz Xj}
\|Bz\|_2^2
\ge \sum_{i \in [n] \setminus J} X_j^2
\quad \text{where} \quad
X_j = \sum_{j \in J} R_{ij} x_j + a'_j.
\end{equation}
By Assumption~\ref{A},  $R_{ij}$ are i.i.d. random variables. Moreover, their distribution satisfies
Assumption~\ref{A: general distribution}, so we can apply Lemma~\ref{lem: SBP simple}
and conclude that for each $j$,
\begin{equation}							\label{eq: weakest SBP}
\Pr{ |X_j| \le c } \le 1 - c'.
\end{equation}

Assume that $\|Bz\|_2^2 \le \a c^2 n$ where $\a \in (0,1)$ is a number to be chosen later.
By \eqref{eq: Bz Xj}, this yields $\sum_{i \in [n] \setminus J} X_j^2 \le \a c^2 n$,
which in turn implies that $X_j \le c$ for at least $|[n] \setminus J| - \a n$ random variables
$X_j$ in this sum. 
Therefore, using independence we obtain
\begin{equation}         \label{eq: Bz via combinations}
  \Pr { \|Bz\|_2^2 \le \a c^2 n }
  \le \binom{|[n] \setminus J|}{\a n} \cdot (1-c')^{|[n] \setminus J|- \a n}
  \le \left( \frac{e}{ \a} \right)^{ \a n} \cdot (1-c')^{(1/2- \a) n}.
\end{equation}
The second inequality holds if we choose $\a$ small enough so that $\a n \le n/4$,
while $|[n] \setminus J| \ge n-N/4 \ge n/2$ by assumption.
The probability bound in \eqref{eq: Bz via combinations} can be made smaller than $e^{-\bar{c} n}$
for some $\bar{c} = \bar{c}(c')>0$ by choosing $\a = \alpha(c') > 0$ sufficiently small.
This completes the proof.
\end{proof}

We are going to argue that the lower bound in Proposition~\ref{prop: Bz simple}
holds not only for a fixed unit vector $z$ but also uniformly over $z \in \Comp$.
This will follow by combining Proposition~\ref{prop: Bz simple} with the following
standard construction on a net for the set of compressible vectors.

\begin{lemma}[Net for compressible vectors]  \label{lem: comp net}
 For any $\d \in (0,1)$, there exists a $(2c_1)$-net of the set of $\Comp$ of cardinality at most
 \[
   \left( \frac{C}{c_0 c_1^2}\right)^{c_0 N}.
 \]
\end{lemma}

\begin{proof}
First we construct a $c_1$-net of the set of sparse vectors.
This set is a union of coordinate subspheres of $S_{\C}^J$ for all sets $J \subset [N]$
of cardinality $c_0 N$. For a fixed $J$, the standard volume argument yields
a $c_1$-net of $S_{\C}^J$ of cardinality at most $(C_1 / c_0 c_1^2)^{c_0 N}$. A union bound
over $\binom{N}{c_0 N} \le C_2^N$ choices of $J$ produces a $c_1$-net of the set of sparse vectors
with cardinality at most $(C / c_0 c_1^2)^{c_0 N}$.
By approximation, this is automatically a $(2c_1)$-net for
the set of compressible vectors.
\end{proof}

\begin{proposition}[A lower bound on the set of compressible vectors]			\label{prop: invertibility on Comp}
  Let $B \in \C^{n \times N}$ be a random matrix satisfying Assumptions~\ref{A} and \ref{A: general distribution}.
  Then one can choose constants $c_0, c_1 \in (0,1)$ in Definition \ref{def: Comp Incomp} depending on
  $p$ and $K$ only, and so that
  $$
  \Pr{ \inf_{z \in \Comp} \|Bz\|_2 \le c \sqrt{n} \text{ and } \BB_{B,M}} \le e^{-cn}.
  $$
\end{proposition}

\begin{proof}
Let us choose $c_1 = c/(4M)$ and let $\NN$ be a $(2 c_1)$-net of the set $\Comp$ given
by Lemma~\ref{lem: comp net}.

Assume the bad event in Proposition~\ref{prop: invertibility on Comp} occurs, so
thus $\|Bz\|_2 \le c \sqrt{n}$ for some $z \in \Comp$ and $\|B\| \le M \sqrt{n}$.
Choose $z_0 \in \NN$ such that $\|z-z_0\|_2 \le 2 c_1$. By triangle inequality, we have
$$
\|Bz_0\|_2 \le \|Bz\|_2 + \|B\| \, \|z-z_0\|_2
\le c \sqrt{n} + M \sqrt{N} \cdot 2c_1
\le 2c \sqrt{n}.
$$
In the last inequality, we used the definition of $c_1$ and the fact that $N \le 2n$.

Furthermore, Proposition \ref{prop: Bz simple} states that for fixed $z_0$, the inequality
$\|Bz_0\|_2 \le 2c \sqrt{n}$ holds with probability at most $e^{-cn}$.
Combining this with the union bound over $z_0 \in \NN$ and using the cardinality of $\NN$
given by Lemma~\ref{lem: comp net}, we conclude that the bad event in
Proposition~\ref{prop: invertibility on Comp} holds with probability at most
$$
e^{-cn} \cdot \left( \frac{C}{c_0 c_1^2}\right)^{c_0 N}.
$$
 Choosing $c_0$ so that the last expression does not exceed $e^{-c_0 n/2}$ completes the proof.
\end{proof}

Proposition~\ref{prop: invertibility on Comp} implies in particular
that with high probability the kernel of $B$ consists of incompressible vectors:
$$
\ker B \cap S^{N-1} \subseteq \Incomp.
$$

\section{Small ball probabilities via real-imaginary correlations}  \label{s: SBP via correlations}

Recall that our big goal is to show that the kernel of $B$ is unstructured, which means that
all vectors in $\ker(B)$ have large LCD. We may try to approach this problem using
the same general line of attack as in Section~\ref{s: kernels incompressible}.
Namely, we can try to bound $\|Bz\|_2$ below
uniformly on the set of vectors with small LCD.

This will require us to considerably sharpen the tools we developed in Section~\ref{s: kernels incompressible} --
small ball probabilities and constructions of nets. More precisely, we would like to make the probability in
Proposition~\ref{prop: Bz simple} exponential in $2n$ rather than $n$; an ideal bound for us would be
\begin{equation}							\label{eq: SBP ideal}
\Pr{ \|Bz\|_2 \le t \sqrt{n} }  \le \Big( Ct + \frac{1}{\sqrt{n}} \Big)^{2n}, \quad t \ge 0.
\end{equation}
A bound like this will be crucial when we combine it with a union bound over a net, just
like in Section~\ref{s: kernels incompressible}. But there the nets were for compressible vectors $z \in \C^{N}$.
Now we will have to handle much larger sets: {\em the level sets of LCD}.
As we will describe in Section~\ref{s: net}, the nets of these level sets are exponential in $2N$.
To control them, it is crucial to have the small probability bound that is also exponential
in $2n$. (The difference between $2N$ and $2n$ is minor
and can intuitively be neglected since $N = (1+\e) n$.)

At first glance, this should be possible because our problem is over $\C$, so the dimension
there should double compared to $\R$. But recall that according to Assumption~\ref{A},
the imaginary part of $B$ is fixed, so there is no extra randomness that could help us
double the exponent.

One can even come up concrete examples where the bound \eqref{eq: SBP ideal} fails.
Assume that the entries of $B$ are {\em real} independent random variables with bounded densities,
and that $z$ is a real vector. Since the matrix $R$ has $n$ rows, the optimal small ball probability is
\begin{equation}         \label{eq: ideal SBP}
\Pr{ \|Bz\|_2 \le t \sqrt{n} }  \le (Ct)^{n}.
\end{equation}
The same is true for complex vectors $z$ with very correlated real and imaginary parts,
such as for $z=x+ix$.

These observations might lead us to the conclusion that it must be impossible
to combine the small ball probabilities with nets.
However, one can notice that the examples of vectors $z$ we just considered are special.
The real vectors $z$ are contained in the $N$-dimensional real sphere, and this sphere
has a net exponential in $N$ rather than $2N$. The same holds for vectors of the type $z=x+ix$.
So these special vectors have smaller nets, which can hopefully be balanced by
the small ball probabilities like \eqref{eq: ideal SBP}.

For other, more ``typical'' vectors, we might hope for stronger probability bounds.
Consider, for example, the vector $z=x+iy$, where $x$ and $y$ have disjoint support and
both have norms $\Omega(1)$. Still assuming that $B$ is a real matrix, we then have
$\|Bz\|_2^2 = \|Bx\|_2^2 + \|By\|_2^2$. The assumption of disjoint support yields implies that
$Bx$ and $By$ are independent, and $\|Bz\|_2^2$ is thus a sum of $2n$ independent
random variables (the row-vector products). So we do have a double amount of randomness here,
and
$$
\Pr{ \|Bz\|_2 \le t \sqrt{n} }  \le ( Ct  )^{2n}.
$$
Such probability bounds can balance a net for the whole sphere of $\C^N$,
which is exponential in $2N$.

Guided by these examples, we may surmise that the small ball probabilities for $Bz$
and the cardinalities of nets for vectors $z$ both depend on the correlation of real
and the imaginary parts of $z$.
Exploring this interaction in search for tight matching bounds for both quantities will
be the main technical difficulty in proving Theorem \ref{thm: distance general}.
We will get a hold of small ball probabilities in the current section, and of cardinalities of nets
in Section \ref{s: net}.


\subsection{Toward a more sensitive bound}

We start by representing $\|Bz\|_2^2$ as a sum of independent random variables
exactly as in Section~\ref{s: Bz as a sum}, leading up to \eqref{eq: Bz double sum restricted}.
In a moment, we will apply Littlewood-Offord theory for each term of
the sum in \eqref{eq: Bz double sum restricted}.
To do this, we express these terms as functions of the rows of $R_{J^c \times J}$ as follows:
\begin{equation}							\label{eq: Bz lower sum}
\|Bz\|_2^2
\ge \sum_{i \in [n] \setminus J} X_j^2 + Y_j^2
= \sum_{i \in [n] \setminus J} \|V_J(R_i)_J - u_i\|_2^2.
\end{equation}
Here
$$
V = \mattwo{x^\tran}{y^\tran} \in \R^{2 \times N}
$$
is a fixed matrix, $R_i^\tran$ denotes the $i$-th row of $R$, and $u_i \in \R^2$ are fixed vectors.

Note that at this time we have three different ways to represent a complex vector $z \in \C^N$:
the usual way $z = x + iy$, as a long real vector $\real{z} = \vectwo{x}{y} \in \R^{2N}$,
and as a $2 \times N$ real matrix $V$ as above.

All $(R_i)_J$ in \eqref{eq: Bz lower sum} are independent real random vectors with all independent coordinates.
We can now apply Theorem~\ref{thm: SBP LCD} in dimension $m=2$ and for $L = 4/\sqrt{p}$.
It yields
\begin{equation}							\label{eq: SBP VRj preliminary}
\Pr{ \|V_J(R_i)_J - u_i\|_2 \le t} \le \frac{C}{\det(V_J V_J^\tran)^{1/2}} \Big( t + \frac{1}{D_2(V_J)} \Big)^2, \quad t \ge 0.
\end{equation}
Here we use the notation $D_2(V_J)$ to emphasize that the least common denominator used in this application
of Theorem~\ref{thm: SBP LCD} is for $2 \times |J|$ matrices, as opposed to the one for vectors which we will focus on later.

\subsection{Disregarding the arithmetic structure}

The small ball probability bound \eqref{eq: SBP VRj preliminary} relies on two different qualities of $z$.
First, the arithmetic structure of $z$ is reflected in the least
common denominator $D_2(V_j)$. Second, the correlation between real and imaginary parts of $z_J$
is measured by the term $\det(V_J V_J^\tran)^{1/2}$.

In this particular place of the argument, we may essentially disregard the arithmetic structure of $z$.
One can get rid of $D_2(V_J)$ using Proposition~\ref{prop: LCD lower simple}, which states that
\begin{equation}							\label{eq: D2 via Linfty}
D_2(V_J) \ge \frac{1}{2\|V_J\|_\infty}.
\end{equation}
To bound $\|V_J\|_\infty$, let us introduce a set of small coordinates as follows.

\begin{definition}[Small coordinates]					\label{def: small}
  Fix $\d \in (0,1)$ and let $z \in \C^N$. We will denote by $\sm(z)$ the set of indices of
  all except the $\d N$ largest (in the absolute value) coordinates of $z$.
  If some of the coordinates of $z$ are equal, the ties are broken arbitrarily.
\end{definition}

If $z$ is a unit vector in $\C^N$ and $J$ is a subset of $\sm(z)$, a simple application of Markov's inequality yields
$\|z_J\|_\infty \le \frac{1}{\sqrt{\delta N}}$. Moreover, by definition of $V$, we have
$\|V_J\|_\infty = \|z_J\|_\infty$. Thus
$$
\|V_J\|_\infty \le \frac{1}{\sqrt{\d N}}.
$$
Substituting this into \eqref{eq: D2 via Linfty}, we conclude that
\begin{equation}							\label{eq: D2 large}
D_2(V_J) \ge \frac{1}{2} \sqrt{\d N}.
\end{equation}
This crude estimate leads to the appearance of the term $1/\sqrt{\e N}$ in Theorem \ref{thm: distance general}. One can probably remove this term by involving the arithmetic structure. However, this would come at  a price of a significant increase of the complexity of the argument, so we did not pursue this direction.

\subsection{Quantifying the real-imaginary correlation}

The determinant $\det(V_J V_J^\tran)^{1/2}$ measures the correlation between real
and imaginary parts of $z_I$. For example, if the real and imaginary parts are equal to each other,
then the determinant vanishes, and the small ball probability bound \eqref{eq: SBP VRj preliminary}
becomes useless.

To make the bound as strong as possible, one would choose the subset $J$ so that,
on the one hand, it lies in $\sm(z)$ to ensure \eqref{eq: D2 large}, and on the other hand,
the determinant $\det(V_J V_J^\tran)^{1/2}$ is maximized. This motivates the following definition.

\begin{definition}[Real-complex correlation]
  For $z \in \C^N$ and $\d \in (0,1)$, we define
  $$
  d(z) = \max \left\{ \det(V_J V_J^\tran)^{1/2} : \; J \subset \sm(z), \; |J| = \d N \right\}.
  $$
\end{definition}

Clearly, $d(z) \in (0,1)$ for any unit vector $z$.

Choosing $J$ that achieves the maximum in the definition of $d(z)$ and using
the bound \eqref{eq: D2 large}, we conclude from \eqref{eq: SBP VRj preliminary}
that
$$
\Pr{ \|V_J(R_i)_J - u_i\|_2 \le t} \le \frac{C}{d(z)} \Big( t + \frac{1}{\sqrt{\d N}} \Big)^2, \quad t \ge 0.
$$
Substituting this into \eqref{eq: Bz lower sum} and using Tensorization Lemma~\ref{lem: tensorization},
we obtain the following result.

\begin{theorem}[Small ball probabilities via real-imaginary correlation]			\label{thm: SBP via correlation}
  Let $B \in \C^{n \times N}$ be a random matrix satisfying Assumptions~\ref{A} and \ref{A: general distribution},
  and let $\d \in (0,1)$. Then for a fixed vector $z \in \C^N$ with $\|z\|_2 = 1$ we have
  $$
  \Pr{\|Bz\|_2 \le t \sqrt{n}}
  \le \left[ \frac{C}{d(z)} \Big( t + \frac{1}{\sqrt{\d n}} \Big)^2 \right]^{(1-\d)n}, \quad t \ge 0.
  $$
\end{theorem}

\subsection{The essentially real case}

Theorem~\ref{thm: SBP via correlation} is useful for vectors $z$ whose real-imaginary correlations $d(z)$
are not too small.
We wonder what could be done in the ``essentially real'' case where $d(z)$ happens to be small?

Our strategy will be different in that case.
Let us first prove a version of Theorem~\ref{thm: SBP via correlation}
that is not based on $d(z)$, but where $t$ is understandably exponential
in $(1-\d)n$ rather than $2(1-\d)n$. Such probability bound will hold for incompressible vectors $z$
(which were introduced Definition~\ref{def: Comp Incomp}), and it will be stronger than the simpler
but more general bound of Proposition~\ref{prop: Bz simple}.

\begin{theorem}[Small ball probabilities for general incompressible vectors]  \label{thm: SBP no correlation}
  Let $B \in \C^{n \times N}$ be a random matrix satisfying Assumptions~\ref{A} and \ref{A: general distribution},
  and let $\d \in (0,1)$. Then for a fixed vector $z \in \Incomp$ we have
  $$
  \Pr{\|Bz\|_2 \le t \sqrt{n}}
  \le \left[ \frac{C}{\sqrt{\d}} \Big( t + \frac{1}{\sqrt{\d n}} \Big) \right]^{(1-\d)n}, \quad t \ge 0.
  $$
\end{theorem}

\begin{proof}
The argument is somewhat simpler than for Theorem~\ref{thm: SBP via correlation}.
Consider the set of small coordinates $\sm(z)$ introduced in Definition~\ref{def: small}.
By definition of that set combined with Markov's inequality, and Definition~\ref{def: Comp Incomp} of incompressible vectors,
we have
$$
\|z_{\sm(z)}\|_\infty \le \frac{1}{\sqrt{\d N}}, \quad \|z_{\sm(z)}\|_2 \ge c.
$$
It follows that there exists a subset $J \subset \sm(z)$ with $|J| = \d N$ and such that
\begin{equation}         \label{eq: zJ}
\|z_J\|_\infty \le \frac{1}{\sqrt{\d N}}, \quad \|z_J\|_2 \ge c \sqrt{\d}.
\end{equation}
(The first inequality is trivial, and the second can be obtained by dividing $\sm(z)$ into
$1/\d-1$ blocks of coordinates of size $\d N$ each, and arguing by contradiction.)

Since $z_J = x_J + i y_J$, either the real part $x_J$ or complex part $y_J$ has $\ell_2$-norm bounded below by
$c\sqrt{\d}/2$.
Let us assume without loss of generality that $x_J$ satisfies this, so
\begin{equation}         \label{eq: xJ}
\|x_J\|_\infty \le \frac{1}{\sqrt{\d N}}, \quad \|x_J\|_2 \ge c' \sqrt{\d}.
\end{equation}

To control $\|Bz\|_2$, we can proceed similarly to the proof of Proposition~\ref{prop: Bz simple},
taking as the starting point the bound
\begin{equation}         \label{eq: Bz Xj again}
\|Bz\|_2^2
\ge \sum_{i \in [n] \setminus J} X_j^2
\quad \text{where} \quad
X_j = \sum_{j \in J} R_{ij} x_j + a'_j.
\end{equation}
For each sum defining $X_j$, we can apply the small ball probability bound of Corollary~\ref{cor: SBP sums}
with $L = \sqrt{8/p}$. This gives
$$
\Pr{|X_j| \le t }
\le \frac{C}{\|x_J\|_2} \Big( t + \frac{1}{D(x_J)} \Big).
$$
We can use the two inequalities in \eqref{eq: xJ} to get rid of the two terms dependent on $x_J$.
Indeed, Proposition~\ref{prop: LCD lower simple} and the first inequality in \eqref{eq: xJ} yield
$$
D(x_J) \ge \frac{1}{2} \sqrt{\d N}.
$$
Using this and the second inequality in \eqref{eq: xJ} gives
$$
\Pr{ |X_j| \le t }
\le \frac{C}{\sqrt{\d}} \Big( t + \frac{1}{\sqrt{\d n}} \Big).
$$
Using this bound for each term of the sum in \eqref{eq: Bz Xj again}
and applying Tensorization Lemma~\ref{lem: tensorization}, we complete the proof.
\end{proof}

Next, we will show that for vectors with small $d(z)$, not only a $\d N$ fraction of coordinates but
almost the entire real and imaginary parts are close to each other. This strong constraint
intuitively means that the set of such vectors is relatively small, and we will indeed construct
a small net for such vectors later.

\begin{lemma}[Real-imaginary correlation]			\label{lem: real-imaginary correlation}
   Let $z \in \C^N$ and set $I := \sm(z)$.
   Then
   $$
   \det(V_I V_I^\tran)^{1/2} \le \frac{C d(z)}{\d}.
   $$
\end{lemma}

\begin{proof}
The argument is based on Cauchy-Binet formula, which yields
\begin{equation}         \label{eq: det VI}
\det(V_I V_I^\tran) = \sum_{I_2 \subset I, \, |I_2|=2} \det(V_{I_2})^2,
\end{equation}
where the sum is over all $\binom{|I|}{2}$ two-element subsets of $I$.
Similarly, for each set $J$ as in the definition of $d(z)$, that is for $J \subset I$, $|J|=\d N$,
we can expand
$$
\det(V_J V_J^\tran) = \sum_{I_2 \subset J, \, |I_2|=2} \det(V_{I_2})^2.
$$
Summing over $J$, we get
$$
\sum_{J \subset I, \, |J| = \d N} \det(V_J V_J^\tran)
= \sum_{J \subset I, \, |J| = \d N} \sum_{I_2 \subset J, \, |I_2|=2} \det(V_{I_2})^2.
$$
To simplify the right hand side, note that every two-element set $I_2 \subset I$ is included
in $\binom{N_0}{\d N - 2}$ sets $J$, where we denote $N_0 := |I| = N - \d N$. Therefore
$$
\sum_{J \subset I, \, |J| = \d N} \det(V_J V_J^\tran)
= \binom{N_0}{\d N - 2} \sum_{I_2 \subset I, \, |I_2|=2} \det(V_{I_2})^2.
$$
The sum in the right hand side equals $\det(V_I V_I^\tran)$ by \eqref{eq: det VI}.
Each determinant $\det(V_J V_J^\tran)$ in the left hand side is bounded by $d(z)^2$ by definition.
This yields
$$
\binom{N_0}{\d N} d(z)^2 \ge \binom{N_0}{\d N - 2} \det(V_I V_I^\tran).
$$
Simplifying this inequality and using that $N_0 = N - \d N$, we complete the proof.
\end{proof}

\section{A net for vectors with given LCD and real-imaginary correlations}		\label{s: net}
Thanks to Section \ref{s: kernels incompressible}, we can now focus on the set of incompressible vectors.
Our goal is to construct a net for the set of incompressible vectors $z$ with given least common
denominator $D(\real{z})$ and real-imaginary correlation $d(z)$. Let us define this set formally.

\begin{definition}[Level set for LCD and real-imaginary correlations]			\label{def: level set}
  For $D, d > 0$, we define with the following subset of $\C^N$:
  $$
  S_{D,d} = \left\{ z \in \Incomp: \; D/2 < D(\real{z}) \le D; \; d(z) \le d \right\}.
  $$
\end{definition}

Hidden in this definition are the parameters $L$ from the definition of $D(\real{z})$ and $\d$ from the definition of $d(z)$,
which we assume to be fixed. When we work with level sets $S_{D,d}$, we can automatically assume that
\begin{equation}         \label{eq: D lower simple}
  D \ge c_0 \sqrt{N},
\end{equation}
since it is relatively easy to see that $D(v) \ge c_0 \sqrt{N}$ for every vector $v \in \Incomp$;
see \cite{V symmetric}.

\medskip

A first attempt at constructing a small $\gamma$-net of the level set $S_{D,d}$ could be to use
the standard volume argument. For instance, if one chooses $\gamma = \sqrt{N} / D$,
the volume argument will yield a net of
cardinality
\begin{equation}							\label{eq: trivial net}
\Big( \frac{D}{\sqrt{N}} \Big)^{2N}.
\end{equation}
The exponent $2N$ appears here because the vectors are complex.

This net is too large for our purposes. Our next, refined, attempt is to leverage the
information about LCD of the vectors in $S_{D,d}$. Indeed, known constructions
lead to the existence of a finer net, namely with $\gamma \ll \sqrt{N}/D$, and still with
approximately the same cardinality as in \eqref{eq: trivial net}, see \cite{RV rectangular}.

However, this net would still be too large if we try to use it in combination with the
small ball probability bound given in Theorem~\ref{thm: SBP via correlation}. Our final,
successful, refinement of the construction will use both LCD and the real-imaginary correlation $d(z)$
of the vectors in $S_{D,d}$. Ideally, we would hope to construct a $\gamma$-net
with $\gamma \ll \sqrt{N}/D$ and with cardinality bounded by
\begin{equation}         \label{eq: ideal net C}
\Big( \frac{D}{\sqrt{N}} \Big)^{2N} d^N.
\end{equation}
The correction term $d^N$ will allow the net to become smaller for more ``real'' vectors --
those with stronger real-imaginary correlations.

Of course, if $d$ is extremely small, such as for purely real vectors, the cardinality
in \eqref{eq: ideal net C} is too good to be true. For purely real vectors, the ideal cardinality
would be the same as in \eqref{eq: trivial net} except with exponent $N$, that is
\begin{equation}         \label{eq: ideal net R}
\Big( \frac{D}{\sqrt{N}} \Big)^N.
\end{equation}

Summarizing, we hope to construct a $\gamma$-net of the level set $S_{D,d}$
for some $\gamma \ll \sqrt{N}/D$, and with cardinality bounded as in \eqref{eq: ideal net R}
if $d$ is not too small (the {\em genuinely complex} case)
as in \eqref{eq: ideal net C}
if $d$ is very small (the {\em essentially real} case).
The following theorem, which is the main result of this section, provides slightly weaker
but still adequate bounds.

\begin{theorem}[Nets for level sets]			\label{thm: nets for level sets}
  There exist constants  $C,c, \bar{c}>0$ such that the following holds.
  Assume that $L$ from the definition of $D(\real{z})$ is such that
  $L \le \bar{c}\sqrt{N}$,
  and $\d$ from the definition of $d(z)$ is such that $\d \in (0,c)$.
  Fix $D>0$, and let
  \begin{equation}							\label{eq: gamma d0}
  \gamma = \frac{L}{D} \sqrt{\log_+ \frac{D}{L}}
  \quad \text{and} \quad
  d_0 = C \d \cdot \max \Big( \gamma, \frac{\sqrt{N}}{D} \Big).
  \end{equation}
  \begin{enumerate}[1.]
    \item (Genuinely complex case). For any $d \ge d_0$,
      there exists a $(C\gamma)$-net of the level set $S_{D,d}$ with cardinality at most
      $$
      \d^{-N} \gamma^{-2 \d N - 1} \Big( \frac{CD}{\sqrt{N}} \Big)^{2N - \d N} d^{N - \d N - 1}.
      $$
    \item (Essentially real case). For any $d \le d_0$,
      there exists a $(C\gamma)$-net of the level set $S_{D,d}$ with cardinality at most
      $$
      \d^{-\d N} \gamma^{-2 \d N - 1} \Big( \frac{CD}{\sqrt{N}} \Big)^{N - \d N + 1}.
      $$
  \end{enumerate}
\end{theorem}

To compare this result with the ideal bounds \eqref{eq: ideal net C} and \eqref{eq: ideal net R},
let us use it with $L \ll \sqrt{N}$. Then $\gamma \ll \sqrt{N}/D$ as we needed, and
the theorem gives bounds similar to \eqref{eq: ideal net C} and \eqref{eq: ideal net R}.

\medskip

We will prove Theorem~\ref{thm: nets for level sets} in the next few subsections.

\subsection{Step 1: setting out the constraints}
We will first construct a net for the points in $S_{D,d}$ with given set of small coordinates;
in the end we unfix this set using the union bound. So let us fix a subset of indices
$$
I \subset [N] \quad \text{with} \quad |I| = N - \d N =: N_0
$$
and define the following subset of $\C^I$:
\begin{equation}         \label{eq: SDdI}
S_{D,d,I} := \left\{ z_I: z \in S_{D,d}, \, \sm(z) = I \right\}.
\end{equation}

Consider a point $z_I \in S_{D,d,I}$. As before, depending on the situation, we will work
with one of the three representations of $z_I$: via real and imaginary parts $z_I = x_I + iy_I$,
via a long real vector $\real{z_I} = \vectwo{x_I}{y_I} \in \R^{2N_0}$, and via the $2 \times N_0$ real matrix
$V_I = \smallmattwo{x_I^\tran}{y_I^\tran}$.

Juxtaposing the available constraints on $z_I$ will help us to construct a small net,
so let us set out precisely what we know about $z_I$. We have three pieces of information.

\medskip

{\em 1. Norm.} We know that $\|z\|_2 = 1$ and $z \in \Incomp$.
Since $|I^c| = \d N \le cN$, the coordinates of $z$ in $I$ must have a significant energy, i.e.
$$
\|z_I\|_2 \ge c.
$$
Since $\|z_I\|_2^2 = \|x_I\|_2^2 + \|y_I\|_2^2$, at least one of these terms is bounded below by $c^2/2$.
Let us assume without loss of generality that it is the first term, which yields
\begin{equation}         \label{eq: constraints norm}
\frac{c}{2} \le \|x_I\|_2 \le 1, \qquad \|y_I\|_2 \le 1.
\end{equation}

\medskip

{\em 2. Least common denominator.}
We know that $D(\real{z}) \in (D/2,D]$. By definition, this implies that there exists $\theta \in [D/2,D]$
and integer points $p,q \in \Z^I$ such that
\begin{equation}         \label{eq: constraints LCD}
\|\theta x_I - p\|_2 \le L \sqrt{\log_+ \frac{\theta}{L}}, \qquad \|\theta y_I - q\|_2 \le L \sqrt{\log_+ \frac{\theta}{L}}.
\end{equation}

\medskip

{\em 3. Real-imaginary correlation.}
We know that $d(z) \le d$. By Lemma~\ref{lem: real-imaginary correlation}, this implies that
$$
\det(V_I V_I^\tran)^{1/2} \le \frac{C d}{\d}.
$$
On the other hand, the determinant is the product of the singular values, that is
$$
\det(V_I V_I^\tran)^{1/2} = s_1(V_I) \, s_2(V_I).
$$
The larger singular value $s_1(V_I)$ is the operator norm of $\|V_I\|$, so it is bounded below by the
norm of either of the two rows of $V_I$. Thus $s_1(V) \ge \|x_I\|_2 \ge c/2$ due to \eqref{eq: constraints norm}.
This gives
\begin{equation}         \label{eq: constraints correlation}
s_2(V_I) \le C' \nu, \quad \text{where} \quad \nu := \frac{C'd}{\d}.
\end{equation}

\subsection{Step 2: an attempt at construction based on LCD}			\label{s: attempt LCD}

Let us ignore for a moment the information about real-imaginary correlation,
and try to construct a net for $S_{D,d,I}$ based on the least common denominator only.
Dividing the inequalities in \eqref{eq: constraints LCD} by $\theta$ and
using that $\theta \ge D/2$, we obtain
\begin{equation}         \label{eq: x approximated}
\Big\| x_I - \frac{p}{\theta} \Big\|_2
\le \frac{L}{\theta} \sqrt{\log_+ \frac{\theta}{L}}
\le \frac{2L}{D} \sqrt{\log_+ \frac{D}{L}}
= 2\gamma,
\end{equation}
and similarly
\begin{equation}         \label{eq: y approximated}
\Big\| y_I - \frac{q}{\theta} \Big\|_2 \le 2\gamma.
\end{equation}
This means that $x_I$ and $y_I$ can be approximated by scaled integer points $p/\theta$
and $q/\theta$, respectively.

To count the integer points $p$ and $q$, let us check their norms.
By triangle inequality, \eqref{eq: constraints LCD} implies
\begin{equation}         \label{eq: p upper}
\|p\|_2 \le \|\theta x_I\|_2 + L \sqrt{\log_+ \frac{\theta}{L}}
\le 2D + L \sqrt{\log_+ \frac{2D}{L}}
\le 3D,
\end{equation}
where we used that $\|x_I\|_2 \le 1$ and $\theta \le D$.

Notice that the bound \eqref{eq: p upper} is sharp within
an absolute constant. Indeed, a similar reasoning gives
$$
\|p\|_2 \ge \|\theta x_I\|_2 - L \sqrt{\log_+ \frac{\theta}{L}}
\ge cD - L \sqrt{\log_+ \frac{D}{L}} \ge c'D.
$$
In the second inequality we used that $\|x_I\|_2 \ge c/2$ due to \eqref{eq: constraints norm}
and $\theta \ge D/2$. In the last inequality, we used that
$D \ge c_0 \sqrt{N}$ due to \eqref{eq: D lower simple} and that $L \le \bar{c} \sqrt{N}$,
choosing $\bar{c}<c_0$ in the formulation of the theorem, which ensures that the term $cD$ dominates.

Similarly to \eqref{eq: p upper}, we obtain
$$
\|q\|_2 \le 3D.
$$
Summarizing, we have shown that $z_I$ can be approximated by a scaled integer point $p+iq$,
where both $p$ and $q$ have norms at most $4D$. Formally, the set
$$
\NN_I := \left\{ \a (p+iq): \; \a \in \R; \; p, q \in \Z^I \cap B(0, 3D) \right\}
$$
is a $(4\gamma)$-net of $S_{D,d,I}$.
How large is this net? Since $D \ge c_0 \sqrt{N}$ due to \eqref{eq: D lower simple},
a standard volume argument shows that the number of integer points in
the real ball $B(0,3D)$ in dimension $|I| = N_0$ is bounded by $(CD/\sqrt{N_0})^{N_0}$.
Thus the number of ``generators'' $p+iq$ of the net $\NN_I$ is bounded by
\begin{equation}         \label{eq: first attempt at net}
\Big( \frac{CD}{\sqrt{N_0}} \Big)^{2N_0}.
\end{equation}
Further, one can easily discretize the multipliers $\a$ (we will do this later),
and obtain a finite net of $S_{D,d,I}$ of cardinality similar to \eqref{eq: first attempt at net}.

The bound \eqref{eq: first attempt at net} is close to the ideal result \eqref{eq: ideal net C}.
However, it misses the $d^N$ factor, which is understandable since we have not used
the real-imaginary correlation $d(z)$ yet.
Let us do this now.

\subsection{Step 3: factoring in the real-imaginary correlation}

Let us rewrite the approximation bound \eqref{eq: constraints LCD} in terms of the
$2 \times N_0$ matrices
$$
V_I = \mattwo{x_I^\tran}{y_I^\tran} \quad \text{and} \quad
W := \mattwo{p^\tran}{q^\tran}.
$$
It follows that $\theta V_I$ is approximated by $W$ in the operator norm:
$$
\|\theta V_I - W\| \le 2 L \sqrt{\log_+ \frac{\theta}{L}}.
$$
Weyl's inequality implies that the corresponding singular values of $\theta V_I$ and $W$
are within $2 L \sqrt{\log_+ \frac{\theta}{L}}$ from each other, and in particular we have
$$
s_2(W) \le s_2(\theta V_I) + 2 L \sqrt{\log_+ \frac{\theta}{L}}
$$
Recalling from \eqref{eq: constraints correlation} that $s_2(V_I) \le C'\nu$ and
that $\theta \le D$, we conclude that
\begin{equation}         \label{eq: s2W}
s_2(W) \le C D \nu + 2 L \sqrt{\log_+ \frac{D}{L}}.
\end{equation}

We can interpret this inequality as saying that the vectors $p$ and $q$ that form the rows of $W$
are almost collinear. Indeed, let $P_{p^\perp}$ denote the orthogonal projection in $\R^I$ onto the subspace
orthogonal to the vector $p$. We claim that
\begin{equation}         \label{eq: almost collinear}
\|P_{p^\perp} q\|_2 \le \Big( 1 + \frac{\|q\|_2}{\|p\|_2} \Big) \, s_2(W).
\end{equation}
To see why this inequality holds, we can express the determinant $\det(W W^\tran)^{1/2}$
in two ways -- via the base times height formula and as the product of singular values:
\begin{equation}         \label{eq: det two ways}
\det(W W^\tran)^{1/2} = \|p\|_2 \cdot \|P_{p^\perp} q\|_2 = s_1(W) \, s_2(W).
\end{equation}
The larger singular value $s_1(W)$ is the operator norm of $W$, which is bounded by the sum of the
norms of the rows:
$$
s_1(W) \le \|p\|_2 + \|q\|_2.
$$
Substituting this into the identity \eqref{eq: det two ways}, we obtain the bound \eqref{eq: almost collinear}.

To successfully apply the bound \eqref{eq: almost collinear}, we recall from Section~\ref{s: attempt LCD}
that $\|q\|_2 \le 3D$ and $\|p\|_2 \ge c'D$, and moreover $s_2(W)$ is bounded as in \eqref{eq: s2W}.
Thus we obtain
\begin{equation}         \label{eq: p q collinear}
\|P_{p^\perp} q\|_2 \le C \Big( D \nu + L \sqrt{\log_+ \frac{D}{L}} \Big).
\end{equation}
Intuitively, this means that $p$ and $q$ are almost collinear, with the degree of collinearity
measured by the real-imaginary correlation factor $d(z)$ (which is reflected here through $\nu = C'd/\delta$).

\subsection{Step 4: construction of the net in the genuinely complex case}

We are now ready to construct a net of $S_{D,d,I}$ based on both LCD and the real-imaginary
correlation. Let us start with the genuinely complex case of the theorem, where
$d \ge d_0$.
Using definitions of $\d_0$ and $\gamma$ in \eqref{eq: gamma d0} and
recalling that $\nu = C'd/\d$, we can rewrite the inequality $d \ge d_0$ as
\begin{equation}         \label{eq: genuinely complex}
D \nu \ge C L \sqrt{\log_+ \frac{D}{L}}
\quad \text{and} \quad
D \nu \ge C \sqrt{N}.
\end{equation}
By the first inequality, the first term dominates in the bound \eqref{eq: p q collinear}, and we have
\begin{equation}         \label{eq: collinearity}
\|P_{p^\perp} q\|_2 \le 2 C D \nu.
\end{equation}

Arguing as in Section~\ref{s: attempt LCD}, we see that the set
$$
\NN_I^{(1)} := \left\{ \a (p+iq): \; \a \in \R; \; p, q \in \Z^I \cap B(0, 3D); \; \|P_{p^\perp} q\|_2 \le 2 C D \nu \right\}
$$
is a $(4\gamma)$-net of $S_{D,d,I}$.
The collinearity condition \eqref{eq: collinearity} included in this definition will allow us to bound the number of
generators $p+iq$ better than before.

First, exactly as in Section~\ref{s: attempt LCD}, the number of possible integer points $p$ in the
definition of $\NN_I^{(1)}$ can be bounded by the standard volume argument, and we have
\begin{equation}         \label{eq: number of p}
\#\left\{ p \text{'s in the definition of } \NN_I^{(1)} \right\}
\le \big| \Z^I \cap B(0, 3D) \big|
\le \Big( \frac{CD}{\sqrt{N_0}} \Big)^{N_0}.
\end{equation}
Next, for a fixed $p$, let us count the number of possible $q$'s that can make a generator
$p+iq$. By definition of $\NN_I^{(1)}$, any such $q$ is an integer point in the cylinder
\begin{equation}							\label{eq: cylinder}
\CC(p, 3D, 2CD\nu)
\end{equation}
where we denote
$$
\CC(p,a,b) =: \left\{ u \in \R^I : \|P_p u\|_2 \le a, \; \|P_{p^\perp} u\|_2 \le b \right\}.
$$
(Here obviously $P_p$ denotes the orthogonal projection in $\R^I$ onto the line
spanned by $p$.)

By a standard covering argument, the number of integer points in the cylinder $\CC(p, a, b)$
is bounded by the volume of the Minkowski sum
$$
\CC(p, a, b) + Q \quad \text{where} \quad Q = \Big[-\frac{1}{2}, \frac{1}{2} \Big]^{N_0}.
$$
Further, the unit cube $Q$ is contained in the Euclidean ball $B(0,\sqrt{N_0})$,
and the Minkowski sum $\CC(p, a, b) + B(0,\sqrt{N_0})$ is clearly contained in the
cylinder
$$
\CC(p, a+\sqrt{N_0}, b+\sqrt{N_0}).
$$
This cylinder is a Cartesian product
of an interval of length $2(a+\sqrt{N_0})$ and the Euclidean ball of radius $b+\sqrt{N_0}$
in the hyperplane orthogonal to the interval, so the volume of the cylinder can be bounded by
\begin{equation}							\label{eq: integer points a b}
2(a+\sqrt{N_0}) \cdot \Big( \frac{C(b+\sqrt{N_0})}{\sqrt{N_0-1}} \Big)^{N_0 - 1}.
\end{equation}

We can apply this bound to our specific cylinder \eqref{eq: cylinder} where $a=3D$ and $b=2CD\nu$.
By \eqref{eq: D lower simple}, $a \ge 4c \sqrt{N_0}$,
and by the second inequality in \eqref{eq: genuinely complex}, $b \ge C \sqrt{N_0}$,
which means that both $\sqrt{N_0}$ terms can be absorbed into $a$ and $b$.
Thus the number of integer points in the cylinder \eqref{eq: cylinder}
is bounded by
\begin{equation}         \label{eq: integer points in cylinder}
Ca \Big( \frac{Cb}{\sqrt{N_0}} \Big)^{N_0-1}
\le CD \Big( \frac{CD\nu}{\sqrt{N_0}} \Big)^{N_0-1}.
\end{equation}

Summarizing, now we know the following about the generators $p+iq$ of the net $\NN_I^{(1)}$.
The number of possible points $p$ is bounded as in \eqref{eq: number of p}.
For each fixed $p$, the number of possible $q$'s that can make the generator $p+iq$
is bounded by the quantity in \eqref{eq: integer points in cylinder}. Thus, the total number of generators
$p+iq$ is bounded by
\begin{equation}         \label{eq: generators genuinely complex}
\Big( \frac{CD}{\sqrt{N_0}} \Big)^{N_0} CD \Big( \frac{CD\nu}{\sqrt{N_0}} \Big)^{N_0-1}.
\end{equation}

\subsection{Step 5: finalizing the genuinely complex case}			\label{s: finalizing complex case}

Three minor points still remain to be addressed in this case.
First, the net $\NN_I^{(1)}$ we constructed
is infinite due to the real multiplier $\alpha$.
Second, this net controls only the coordinates that lie in $I$
(recall the definition \eqref{eq: SDdI} of the set $S_{D,d,I}$).
Third, the construction we made was for a fixed set of coordinates $I$.
We will now take care of these issues.

\subsubsection{Discretizing the multipliers}
The first point can be addressed by discretizing the set of multipliers $\alpha$ in the
definition of the net $\NN_I^{(1)}$. Since
$S_{D,d,I}$ is a subset of the unit ball $B(0,1)$,
we may consider only the multipliers $\alpha$ that satisfy $\|\a(p+iq)\|_2 \le 1$.
For a fixed generator $p+iq$,
we discretize the interval of multipliers $\{ \alpha \in \R: \|\a(p+iq)\|_2 \le 1\}$
by replacing it with a set of $2/\gamma$ numbers $\alpha_j$ that are equally spaced
in that interval. The vector $\a(p+iq)$ can then be approximated by a vector
$\a_i(p+iq)$ with error at most $\gamma$ in the Euclidean norm.

Since $\NN_I^{(1)}$ is a $(4\gamma)$-net of $S_{D,d,I}$, the discretization we just
constructed is a $(6\gamma)$-net of $S_{D,d,I}$. Let us call this net $\MM_I^{(1)}$.
The cardinality of $\MM_I^{(1)}$ is bounded by $(2/\gamma)$
times the number of generators of $\NN_I^{(1)}$, which we bounded in \eqref{eq: generators genuinely complex}.
In other words,
\begin{equation}         \label{eq: MM}
|\MM_I^{(1)}|
\le \frac{2}{\gamma} \cdot \Big( \frac{CD}{\sqrt{N_0}} \Big)^{N_0} CD \Big( \frac{CD\nu}{\sqrt{N_0}} \Big)^{N_0-1}.
\end{equation}

\subsubsection{Controlling the coordinates outside $I$}
The second point we need to address is that since $\MM_I^{(1)}$ is a $(6\gamma)$-net of the set
$S_{D,d,I} = \left\{ z_I: z \in S_{D,d}, \, \sm(z) = I \right\}$, this net can only control the coordinates of $z$
in $I$. To control the coordinates outside $I$, it is enough to construct a separate $\gamma$-net for the set
$$
\left\{ z_{I^c}: z \in S_\C^{N-1} \right\}.
$$
Since $|I^c| = \d N$, a standard volume bound allows one to find such a net of cardinality at
most $(5/\gamma)^{2\d N}$. Combining the two nets, we conclude that there exists a $(7\gamma)$-net
of the set
$$
T_{D,d,I} = \left\{ z: z \in S_{D,d}, \, \sm(z) = I \right\}
$$
of cardinality at most
$$
\Big( \frac{5}{\gamma} \Big)^{2\d N} |\MM_I^{(1)}|.
$$

\subsubsection{Unfixing the set of coordinates $I$}

Finally, the third point we need to address is that our construction was for a fixed set of indices $I$.
To unfix $I$, we note that there is at most
$$
\binom{N}{N-|I|} = \binom{N}{\d N} \le \Big( \frac{e}{\delta} \Big)^{\d N}
$$
ways to choose $I$. So, combining the nets we constructed for each $I$, we obtain a $(7\gamma)$-net
of $S_{D,d}$ of cardinality at most
$$
\Big( \frac{e}{\delta} \Big)^{\d N} \cdot \Big( \frac{5}{\gamma} \Big)^{2\d N} |\MM_I^{(1)}|.
$$
Substituting here the bound \eqref{eq: MM} for $|\MM_I^{(1)}|$,
recalling that $N_0 = N -\d N$ and $\nu = C'd/\d$,
and simplifying the expression, we prove the first part of the theorem.

\subsection{Step 5: the essentially real case}

We proceed to the essentially real case, where $d < d_0$.
This means that at least one of the inequalities in \eqref{eq: genuinely complex} fails.

\subsubsection{Case 1} Assume that the first inequality in \eqref{eq: genuinely complex} holds
but the other fails, that is
$$
D \nu \ge C L \sqrt{\log_+ \frac{D}{L}}
\quad \text{and} \quad
D \nu \le C \sqrt{N}.
$$
We proceed in the same way as in the genuinely complex case until we apply the general
bound on the integer points \eqref{eq: integer points a b} to our cylinder with $a = 4D$ and
$b = 2CD\nu$. This is the only place where we used the second inequality in \eqref{eq: genuinely complex},
which now fails. This means that $b$ gets absorbed into the $\sqrt{N_0}$ term, and the
the number of integer points in the cylinder \eqref{eq: cylinder}
is consequently bounded by
$$
Ca \Big( \frac{C\sqrt{N_0}}{\sqrt{N_0}} \Big)^{N_0-1}
\le C^N D.
$$
Using this bound in place of \eqref{eq: integer points in cylinder} and arguing
exactly as in the genuinely complex case, we complete the proof for this sub-case.

\subsubsection{Case 2} The remaining sub-case is where the first inequality in \eqref{eq: genuinely complex}
fails, that is
\begin{equation}         \label{eq: essentially real}
D \nu < L \sqrt{\log_+ \frac{D}{L}}.
\end{equation}
Then the second term dominates in the bound \eqref{eq: p q collinear}, and we have
\begin{equation}         \label{eq: collinearity real}
\|P_{p^\perp} q\|_2 \le 2 C L \sqrt{\log_+ \frac{D}{L}}.
\end{equation}

Let us fix a points $z_I = x_I+iy_I$ from $S_{D,d,I}$.
Since the orthogonal projection has norm one, \eqref{eq: y approximated}
yields
$$
\Big\| P_{p^\perp} \Big( y_I - \frac{q}{\theta} \Big) \Big\|_2 \le \gamma.
$$
Combining this with \eqref{eq: collinearity real} and using triangle inequality, we obtain
$$
\|P_{p^\perp} y_I\|_2 \le \gamma + \frac{2 C L}{\theta} \sqrt{\log_+ \frac{D}{L}}.
$$
Recalling that $\theta \ge D$ and the definition of $\gamma$ in the theorem, we obtain
$$
\|P_{p^\perp} y_I\|_2 \le C\gamma.
$$

We can interpret this inequality as follows. There exists a multiplier $\beta \in \R$ such that
$$
\|y_I - \beta p\|_2 \le C\gamma.
$$
Let us rewrite \eqref{eq: x approximated} in a similar way -- there exists a multiplier $\alpha = 1/\theta \in \R$
such that
$$
\|x_I - \alpha p\|_2 \le \gamma.
$$
Recalling that $z_I = x_I + iy_I$, it follows that
$$
\|z_I - (\alpha+i\beta) p\|_2 \le C\gamma.
$$
Furthermore, recalling from \eqref{eq: p upper} that $\|p\|_2 \le 3D$, we conclude that the set
$$
\NN_I^{(2)} := \left\{ (\alpha+i\beta) p: \; \alpha,\beta \in \R; \; p \in \Z^I \cap B(0, 3D) \right\}
$$
is a $(C\gamma)$-net of $S_{D,d,I}$.

The number of generators $p$ can be counted by a volume argument, exactly as in \eqref{eq: number of p}:
$$
\#\left\{ p \text{'s in the definition of } \NN_I^{(2)} \right\}
\le \big| \Z^I \cap B(0, 3D) \big|
\le \Big( \frac{CD}{\sqrt{N_0}} \Big)^{N_0}.
$$
Finally, we can discretize the multipliers $\a+i\beta$ and unfix the set $I$ similarly to how we did it
in Section~\ref{s: finalizing complex case}. We obtain a $(6\gamma)$-net of $S_{D,d}$
of cardinality at most
\begin{equation}							\label{eq: number of p real}
\Big( \frac{e}{\delta} \Big)^{\d N} \cdot \Big( \frac{5}{\gamma} \Big)^{2\d N}
\cdot \Big( \frac{C}{\gamma} \Big)^2 \cdot \Big( \frac{CD}{\sqrt{N_0}} \Big)^{N_0}.
\end{equation}
(To recall, the first term here comes from unfixing $I$, the second from controlling coordinates outside $I$,
the third from discretizing the multipliers in the complex disc $\{ \a+i\beta \in \C: \|(\alpha+i\beta) p\|_2 \le 1 \}$,
and fourth from \eqref{eq: number of p real}.)

Recalling that $N_0 = N -\d N$ and $\nu = C'd/\d$,
and simplifying the expression, we prove the second part of Theorem~\ref{thm: nets for level sets}.
\qed

\section{Structure of kernels, and proof of Theorem~\ref{thm: distance general} on the distances}	\label{s: distance proof}
In this section we will show that random subspaces, and specifically the kernels of random matrices,
are arithmetically unstructured, which means that they have large LCD with high probability.
The following is the main result of this section.

\begin{theorem}[Kernels of random matrices are unstructured]				\label{thm: kernel unstructured}
  Let $B \in \C^{n \times N}$ be a random matrix satisfying Assumptions~\ref{A} and \ref{A: general distribution},
  and assume that $n = (1-\e) N$ for some $\e \in (2/n,c)$.
  Set $L := \sqrt{\e N}$.
  Then,
  $$
  \Pr{D(\real{\ker B}, L) \le \min \left( \sqrt{N} e^{c/\sqrt{\e}}, \; \e N \right) \ \textrm{and} \ \BB_{B,M} } \le e^{-cN}.
  $$
\end{theorem}

This theorem will follow by balancing the two  forces -- the small ball probability estimates
of Section~\ref{s: SBP via correlations} and the net for vectors with given LCD of Section~\ref{s: net}.

It would be convenient to first state a preliminary version of Theorem~\ref{thm: kernel unstructured}
which holds for vectors with given levels of LCD $D(\real{z})$ and real-imaginary correlation $d(z)$.
We will work here with somewhat smaller level sets than $S_{D,d}$ from Section~\ref{s: net}.
For $D,d>0$, we consider
$$
\bar{S}_{D,d} = \left\{ z \in \Incomp: \; D/2 < D(\real{z}) \le D; \; d/2 \le d(z) \le d \right\}.
$$
Clearly, $S_{D,d}$ is the union of the sets $\bar{S}_{D,d}$ for all $d \le d_0$.

\begin{proposition}[Kernels and level sets]				\label{prop: kernels and level sets}
  Let $B \in \C^{n \times N}$ be a random matrix satisfying Assumptions~\ref{A} and \ref{A: general distribution},
  and let $n = (1-\e) N$ for some $\e \in (2/n,c)$.
  Set $L$ from the definition of $D(\real{z})$ to be $L := \sqrt{\e N}$.
  Let
  $$
  D \le \min \left( \sqrt{N} e^{c/\sqrt{\e}}, \; \e N \right)
  $$
  and let $d_0$ be the threshold value from \eqref{eq: gamma d0} for $\d = c \sqrt{\e}$.
  \begin{enumerate}[1.]
    \item (Genuinely complex case). For any $d \in [d_0,1]$, we have
      \begin{equation}         \label{eq: kernels and level sets C}
      \Pr{ \bar{S}_{D,d} \cap \ker B \ne \emptyset  \ \textrm{and} \ \BB_{B,M} } \le e^{-N}.
      \end{equation}
    \item (Essentially real case). For any $d \in [0, d_0]$, we have
      $$
      \Pr{ S_{D,d} \cap \ker B \ne \emptyset  \ \textrm{and} \ \BB_{B,M} } \le e^{-N}.
      $$
  \end{enumerate}
\end{proposition}
Note here that in the genuinely complex case, we use an additional stratification by $d(z)$ by considering the sets $\bar{S}_{D,d}$, while in the essentially real case, we treat the set $S_{D,d}$ in one strike.

We will prove this proposition in the next two subsections.

\subsection{Proof of Proposition~\ref{prop: kernels and level sets} in the genuinely
complex case}

\subsubsection{Step 1: combining the small ball probability with the net}

We fix a vector $z \in \bar{S}_{D,d}$ and apply Theorem~\ref{thm: SBP via correlation}.
By the assumptions on $n$, we can write the conclusion of this theorem as follows:
\begin{equation}         \label{eq: Bz SBP}
\Pr{\|Bz\|_2 \le t \sqrt{N}}
\le \left[ \frac{C}{d} \Big( t + \frac{1}{\sqrt{\d N}} \Big)^2 \right]^{(1-\d)(1-\e)N}, \quad t \ge 0.
\end{equation}
Let us apply this bound for $t := \l \sqrt{\d}$, where $\l \in (0,1)$ is a parameter
whose value we choose later. If we assume that
\begin{equation}         \label{eq: lambda restriction 1}
\l \sqrt{\d} \ge \frac{1}{\sqrt{\d N}},
\end{equation}
then the probability bound can be expressed as
$$
\Pr{\|Bz\|_2 \le \l \sqrt{\d N}}
\le \Big( \frac{C \l^2 \d}{d} \Big)^{(1-\d)(1-\e)N}.
$$

Next, Theorem~\ref{thm: nets for level sets} provides us with a $(C\gamma)$-net of $S_{D,d}$
of controlled cardinality. Clearly, the same is true for the smaller set $\bar{S}_{D,d}$. Let us
denote such a net by $\NN$. Using a union bound, we can combine the probability bound
with the net as follows:
$$
p:= \Pr{\inf_{z \in \NN} \|Bz\|_2 \le \l \sqrt{\d N}}
\le \Big( \frac{C \l^2 \d}{d} \Big)^{(1-\d)(1-\e)N} \cdot |\NN|.
$$
recalling the bound on $|\NN|$ given by Theorem~\ref{thm: nets for level sets}, we obtain
\begin{equation}         \label{eq: p long bound}
p \le \Big( \frac{C \l^2 \d}{d} \Big)^{(1-\d)(1-\e)N}
\cdot \d^{-N} \gamma^{-2 \d N - 1} \Big( \frac{CD}{\sqrt{N}} \Big)^{2N - \d N} d^{N - \d N - 1}.
\end{equation}
Our goal is to show that $p \le e^{-N}$.

\subsubsection{Step 2: simplifying the probability bound}

Assume that $\d$ (which we recall is a parameter from the definition of $d(z)$) is chosen such that
\begin{equation}         \label{eq: epsilon delta}
2\e \le \d \le \sqrt{\e}.
\end{equation}
We claim that $d$, $\d$ and $\gamma$ can be removed from \eqref{eq: p long bound}
at the cost of increasing the bound by $C^N$.

To see this for $d$, note that the exponent of $d$ in the bound \eqref{eq: p long bound} is
$$
(N - \d N - 1) - (1-\d)(1-\e)N = (1-\d) \e N - 1 \ge \e N / 2 - 1 \ge 0.
$$
Since $d \in [0,1]$, removing $d$ can only make the bound \eqref{eq: p long bound} larger.

Similarly, the exponent of $1/\d$ in the bound \eqref{eq: p long bound} is
$$
N - (1-\d)(1-\e)N \le 2 \d N.
$$
Thus $\d$ contributes to the bound a factor not larger than $\d^{-2\d N} \le C^N$.

Finally, to evaluate the contribution of $\gamma$, let us denote
$$
\bar{D} : = \frac{D}{L} = \frac{D}{\sqrt{\e N}}.
$$
Notice in passing that $\bar{D} \ge e$ since $D \ge c_0 \sqrt{N}$ by \eqref{eq: D lower simple}.
Recalling the definition of $\gamma$ in \eqref{eq: gamma d0}, we have
$$
\frac{1}{\g} = \frac{\bar{D}}{\sqrt{\log_+\bar{D}}} \le \bar{D}.
$$
Therefore, the contribution of $\g$ to the bound \eqref{eq: p long bound} is a factor
not larger than
$$
\g^{-2 \d N - 1} \le \bar{D}^{4 \d N}.
$$
To bound this quantity further, we can use the assumption on $D$ and the comparison \eqref{eq: epsilon delta},
which imply
\begin{equation}         \label{eq: Dbar small}
\bar{D} \le \frac{1}{\sqrt{\e}} \, e^{c/\sqrt{\e}}
\le \frac{1}{\d} \, e^{c/\d}.
\end{equation}
This implies that $\bar{D}^{4 \d N} \le C^N$.
Therefore, $\g$ contributes to the bound \eqref{eq: p long bound} a factor
not larger than $C^N$.

We have shown that $d$, $\d$ and $\gamma$ can be removed from \eqref{eq: p long bound}
at the cost of increasing the bound by $C^N$. In other words, we have
$$
p \le \left[ C \l^{2(1-\d)(1-\e)} \left( C \sqrt{\e} \bar{D} \right)^{2-\d} \right]^N.
$$

Solving for $\l$, we see that the desired bound
$$
p \le e^{-N}
$$
holds whenever
$$
\l \le \Big( \frac{c}{\sqrt{\e} \bar{D}} \Big)^{\frac{2-\d}{2(1-\e)(1-\d)}},
$$
which in turn holds if $\l$ is chosen so that
\begin{equation}         \label{eq: lambda restriction 2}
\l \le \Big( \frac{c}{\sqrt{\e} \bar{D}} \Big)^{1+2\d}.
\end{equation}

\subsubsection{Step 3: approximation by the net}

In the previous step we showed that the event
$$
\inf_{z_0 \in \NN} \|Bz_0\|_2 \le \l \sqrt{\d N}
$$
holds with probability at least $1-e^{-N}$, as long as the parameter $\l$ satisfies
\eqref{eq: lambda restriction 1} and \eqref{eq: lambda restriction 2}.
Let us fix a realization of the random matrix $B$ for which this event does hold.

Fix a vector $z \in \bar{S}_{D,d}$. To finish the proof of \eqref{eq: kernels and level sets C},
we need
to show that $\|Bz\|_2 > 0$. Let us choose $z_0 \in \NN$ which best approximates
the vector $z$; by definition of $\NN$ we have
$$
\|z-z_0\|_2 \le C\gamma = \frac{C \sqrt{\log_+ \bar{D}}}{\bar{D}}.
$$
Assume that the event $\BB_{B,M}$ occurs.
By triangle inequality, it follows that
\begin{align*}
\|Bz\|_2
  &\ge \|Bz_0\|_2 - \|B\| \cdot \|z-z_0\|_2 \\
  &\ge \l \sqrt{\d N} - M \sqrt{N} \cdot \frac{C \sqrt{\log_+ \bar{D}}}{\bar{D}}
\end{align*}
This quantity is positive as we desired if $\l$ satisfies
\begin{equation}         \label{eq: lambda restriction 3}
\l \ge \frac{C \sqrt{\log_+ \bar{D}}}{\sqrt{\d} \bar{D}}.
\end{equation}
Recall thatwe allow our constants to depend on $M$. This allows to absorb $M$ in $C$ in the inequality above.

\subsubsection{Step 4: final choice of the parameters}

We have shown that the conclusion \eqref{eq: kernels and level sets C}
of the proposition in the genuinely complex case holds if we can choose the parameters $\d$ and $\l$
in such a way that they satisfy \eqref{eq: epsilon delta}, \eqref{eq: lambda restriction 1},
\eqref{eq: lambda restriction 2} and \eqref{eq: lambda restriction 3}.
We will now check that such a choice indeed exists.

First, let us choose $\l$ just large enough to satisfy \eqref{eq: lambda restriction 3}; thus we set
$$
\l := \frac{C \sqrt{\log_+ \bar{D}}}{\sqrt{\d} \bar{D}}.
$$
To check \eqref{eq: lambda restriction 1}, we use the assumption that $D \le \e N$,
which implies that $\bar{D} \le \sqrt{\e N}$. This and the choice of $\l$ imply
$$
\l \sqrt{\d} \ge \frac{1}{\bar{D}} \ge \frac{1}{\sqrt{\e N}} \ge \frac{1}{\sqrt{\d N}},
$$
where the last inequality follows from \eqref{eq: epsilon delta}.
This proves \eqref{eq: lambda restriction 1}.

It remains to check \eqref{eq: lambda restriction 2}, which takes the form
$$
\frac{C \sqrt{\log_+ \bar{D}}}{\sqrt{\d} \bar{D}}
\le \Big( \frac{c}{\sqrt{\e} \bar{D}} \Big)^{1+2\d}.
$$
Rearranging the terms and dropping the term $(1/\sqrt{\e})^{2\d}$
which is smaller than an absolute constant due to \eqref{eq: epsilon delta},
we may rewrite this restriction as
$$
\bar{D}^{4\d} \log_+ \bar{D} \le \frac{c\d}{\e}.
$$
Substituting here the assumption $\bar{D} \le \frac{1}{\d} \, e^{c/\d}$ which we already
used in \eqref{eq: Dbar small}, we see that the restriction is satisfied if
$$
\frac{C}{\d} \le \frac{c\d}{\e}.
$$
It remains to choose $\d := c \sqrt{\e}$; then the restriction is satisfied,
and we have verified \eqref{eq: lambda restriction 2}.
This finished the proof of the proposition in the genuinely complex case.

\subsection{Proof of Proposition~\ref{prop: kernels and level sets} in the essentially
real case}

The argument is similar, and even simpler, than in the genuinely complex case.
We just need to use the appropriate small ball probability bound, namely
Theorem~\ref{thm: SBP no correlation} instead of \eqref{eq: Bz SBP},
and the corresponding bound on the net -- the one from the essentially real case
in Theorem~\ref{thm: nets for level sets}. This leads to the following variant
of \eqref{eq: p long bound}:
$$
p \le (C\l)^{(1-\d)(1-\e)N}
\cdot \d^{-\d N} \gamma^{-2 \d N - 1} \Big( \frac{CD}{\sqrt{N}} \Big)^{N - \d N + 1}.
$$

Like before, we can remove $\d$ and $\g$, simplifying the bound to
$$
p \le \left[ C \l^{(1-\d)(1-\e)} \left( C \sqrt{\e} \bar{D} \right)^{1-\d} \right]^N.
$$
Then the desired bound $p \le e^{-N}$ holds if $\l$ is chosen so that
$$
\l \le \Big( \frac{c}{\sqrt{\e} \bar{D}} \Big)^{\frac{1}{1-\e}}.
$$
In particular, this holds if $\l$ satisfies the same restriction as in the
genuinely complex case, namely \eqref{eq: lambda restriction 2}.
(To see this, note that $1+2\d \ge 1/(1-\e)$ by \eqref{eq: epsilon delta}.)

The rest of the proof is exactly as in the genuinely complex case.
Proposition~\ref{prop: kernels and level sets} is proved.
\qed

\subsection{Proof of Theorem~\ref{thm: kernel unstructured}}

For convenience, let us denote
$$
D_0 := \min(\sqrt{N} e^{c/\sqrt{\e}}, \; \e N).
$$
Assume that $D(\real{\ker B}, L) \le D_0$, and the event $\BB_{B,M}$ occurs. This means that there exists $z \in S^{N-1}$ such that
$$
z \in \ker B, \quad D(\real{z}) \le D_0.
$$

We can bound the probability of this event by considering the following cases.
If $z$ is compressible, then such event holds with probability at most $e^{-c_1 N}$ according
to Proposition~\ref{prop: invertibility on Comp}. Assume that $z$ is incompressible.
In the genuinely complex case where $d(\real{z}) > d_0$, the vector $z$ must belong to
a level set $\bar{S}_{D,d}$ for some $D \in [c_0 \sqrt{N}, D_0]$ and $d \in [d_0,1]$.
(The lower bound on $D$ here is from \eqref{eq: D lower simple}.)
For given $D$ and $d$, the probability that such $z$ exists is at most $e^{-c_2 N}$
by the first part of Proposition~\ref{prop: kernels and level sets}.
In the remaining, essentially real case where $d(\real{z}) \le d_0$, the vector $z$ must belong to
a level set $S_{D,d_0}$ for some $D \in [c_0 \sqrt{N}, D_0]$.
For given $D$ and $d$, the probability that such $z$ exists is at most $e^{-c_2 N}$
by to the second part of Proposition~\ref{prop: kernels and level sets}.

This reasoning shows that the probability that $D(\real{\ker B}, L) \le D_0$ is bounded by
\begin{equation}         \label{eq: prob double sum}
e^{-c_1 N}
+ \sum_{\substack{D \in [c_0 \sqrt{N}, D_0] \\ \text{dyadic}}} \; \sum_{\substack{d \in [d_0, 1] \\ \text{dyadic}}} e^{-c_2 N}
+ \sum_{\substack{D \in [c_0 \sqrt{N}, D_0] \\ \text{dyadic}}} e^{-c_3 N}.
\end{equation}
The definitions of the level sets allowed us here to discretize the ranges of $D$ and $d$
by including only the dyadic values in the sum, namely the values of the form $2^k$, $k \in \Z$.

It remains to bound the number the terms in the sums.
The number of dyadic values in the interval $[c_0 N, D_0]$ is at most
$$
\log \Big( \frac{D_0}{c_0 \sqrt{N}} \Big) \le \log N
$$
since $D_0 \le N$ by definition.
Similarly, using the definition~\ref{eq: gamma d0} of $d_0$, we see that the number of dyadic values
in the interval $[d_0,1]$ is at most
$$
\log \Big( \frac{1}{d_0} \Big)
\le \log \Big( \frac{D}{\e N} \Big)
\le \Big( \frac{D_0}{\e N} \Big)
\le \log N.
$$
Therefore, the probability estimate \eqref{eq: prob double sum} is bounded by
$$
e^{-c_1 N} + \log^2(N) e^{-c_2 N} + \log(N) e^{-c_3 N} \le e^{-cN}.
$$
This completes the proof.
\qed

\subsection{Distances between random vectors and subspaces}

We are ready to prove Theorem~\ref{thm: distance general} on the distances between random vectors and subspaces.
It can be quickly deduced by combining small ball probability bounds we developed in Section~\ref{s: LCD and SBP}
with the bound on LCD for random subspaces, namely Theorem~\ref{thm: kernel unstructured}.

Given a random vector $Y \in \R^k$ and an event $\Omega$, denote
\[
  \LL_{\Omega} (Y, r) = \sup_{y \in \R^k} \Pr{\|Y-y\|_2 <r \ \text{and} \ \Omega}.
\]
In Section~\ref{s: from C to R general}, we reduced Theorem~\ref{thm: distance general} to a problem over reals.
Indeed, according to \eqref{eq: reduction of dist to R}, it suffices to bound
$$
p_0 := \LL_{\BB_{B,M}}(P_{\real{\ker B}} \widehat{Z}, 2 \tau \sqrt{\e N}).
$$
For this, we recall that $\text{dim}(\real{\ker B}) =2 \e N$ and apply Corollary~\ref{cor: SBP proj}, which yields
\begin{equation}         \label{eq: LL via SBP}
p_0 \le \Big( \frac{CL}{\sqrt{\e N}} \Big)^{2 \e N}
  \Big( \tau + \frac{\sqrt{\e N}}{D(\real{\ker B},L)} \Big)^{2 \e N}.
\end{equation}
Next, Theorem~\ref{thm: kernel unstructured} states that for $L = \sqrt{\e N}$, with probability
at least $1-e^{-cN}$ we have
$$
D(\real{\ker B},L) \ge \min \left( \sqrt{N} e^{c/\sqrt{\e}}, \; \e N \right).
$$
Substituting this into \eqref{eq: LL via SBP} and simplifying the bound, we obtain
$$
p_0
\le \left[ C \Big( \tau + \frac{1}{\sqrt{\e N}}
+ \sqrt{\e} e^{-c/\sqrt{\e}}\Big) \right]^{2 \e N}
+ e^{-c  N}
\le \left[ C \Big( \tau + \frac{1}{\sqrt{\e N}} + e^{-c'/\sqrt{\e}} \Big) \right]^{2  \e N}.
$$
This inequality combined with \eqref{eq: reduction of dist to R} completes the proof of Theorem~\ref{thm: distance general}.
\qed

\section{Proof of Theorem~\ref{thm: invertibility general} on invertibility for general distributions}	\label{s: general proof}

The strategy of the proof of Theorem~\ref{thm: invertibility general} will be very close
to the argument we gave for continuous distributions in Section~\ref{s: continuous}.
However, there are two important differences.
First, the distance bound for continuous distributions
given in Lemma~\ref{lem: distance} can not hold for general distributions;
we will replace it by Theorem~\ref{thm: distance general}.
Another ingredient that is not available for general distributions is the lower bound
given in Lemma~\ref{lem: fixed row and vector} and all of its consequences in Section~\ref{s: G below on E-}.
Instead, we will use the small ball probability bounds for general distributions that we developed
in Section~\ref{s: special cases sums}. Let us start with this latter task.

\subsection{$G$ is bounded below on the small subspace $E^-$}

Here we extend the argument of Section~\ref{s: G below on E-} to general distributions.

\begin{lemma}			\label{lem: E- Incomp}
  With probability at least $1-e^{-cn}$,
  we have $S_{E^-} \subset \Incomp$.
\end{lemma}

\begin{proof}
The definition of $E^-$ in Section~\ref{s: E+} implies that
$$
\|Bz\|_2 \le c \tau \e \sqrt{n} \quad \text{for all } z \in S_{E^-}.
$$
On the other hand, Proposition~\ref{prop: invertibility on Comp} states that
with probability at least  $1-e^{-cn}$,
$$
\|Bz\|_2 > c \sqrt{n} \quad \text{for all } z \in \Comp.
$$
Since these two bounds can not hold together for the same $z$, it follows
that the sets $S_{E^-}$ and $\Comp$ are disjoint. This proves the lemma.
\end{proof}

The following result is a version of Lemma~\ref{lem: fixed row and vector} for general distributions.

\begin{lemma}[Lower bound for a fixed row and vector]			\label{lem: fixed row and vector general}
  Let $G_j$ denote the $j$-th row of $G$.
  Then for each $j$, $z \in \Incomp$, and $\theta \ge 0$, we have
  $$
  \Pr{ |\ip{G_j}{z}| \le \theta} \le C \Big( \theta + \frac{1}{\sqrt{n}} \Big).
  $$
\end{lemma}

\begin{proof}
Fix $z = x + iy$ and let $J$ be the set of indices of all except $cn$ largest (in the absolute value) coordinates of $z$.
By Markov's inequality and Definition~\ref{def: Comp Incomp} of incompressible vectors,
we have
$$
\|z_J\|_\infty \le \frac{c}{\sqrt{n}}, \quad \|z_J\|_2 \ge c.
$$
Since $z_J = x_J + i y_J$, either the real part $x_J$ or the complex part $y_J$
has $\ell_2$-norm bounded below by $c/2$.
Let us assume without loss of generality that $x_J$ satisfies this, so
\begin{equation}							\label{eq: xJ Linfty L2}
\|x_J\|_\infty \le \frac{c}{\sqrt{n}}, \quad \|x_J\|_2 \ge c.
\end{equation}
The first inequality and Proposition~\ref{prop: LCD lower simple} imply that
$$
D_1(x_J) \ge c \sqrt{n}.
$$
Here we use notation $D_1(\cdot)$ for the LCD in dimension $m=1$ and with $L \sim 1$;
note that it is distinct from the LCD in dimension $m=2$ we studied in the major part of this paper.

We proceed similarly to the proof of Lemma~\ref{lem: fixed row and vector}.
Decomposing the random vector $Z := G_j$ as $Z = X + i Y$, we obtain the bound \eqref{eq: SBP Zz}:
$$
\Pr{ |\ip{Z}{z}| \le \theta} \le \LL( \ip{X}{x}, \theta ).
$$
Further, we use the restriction property of the concentration function (Lemma~\ref{lem: SBP restriction})
followed by the small ball probability bound (Corollary~\ref{cor: SBP sums}), and obtain
$$
\LL( \ip{X}{x}, \theta )
\le \LL( \ip{X_J}{x_J}, \theta )
\le C \Big( \theta + \frac{1}{\sqrt{n}} \Big).
$$
This completes the proof of the lemma.
\end{proof}

Using Tensorization Lemma~\ref{lem: tensorization} exactly as we did before in Section~\ref{s: lower on fixed vector},
we obtain the following version of Lemma~\ref{lem: fixed vector} for general distributions.

\begin{lemma}[Lower bound for a fixed vector]		\label{lem: fixed vector general}
  For each $x \in \Incomp$ and $\theta >0$, we have
  $$
  \Pr{ \|Gx\|_2 \le \theta \sqrt{\e n} } \le \Big( C\theta + \frac{C}{\sqrt{n}} \Big)^{\e n}.
  $$
\end{lemma}

In particular, if $\theta \ge 1/\sqrt{n}$, then the probability is further bounded by $(C\theta)^{\e n}$.
This is similar to the bound we had in Lemma~\ref{lem: fixed vector} for the continuous case.
Using this observation, we deduce the following version of Lemma~\ref{lem: G bounded below on E}
for general distributions, and with the same proof.

\begin{lemma}[Lower bound on a subspace]		\label{lem: G bounded below on E general}	
  Let $M \ge 1$ and $\mu \in (0,1)$.
  Let $E$ be a fixed subspace of $\C^n$ of dimension at most $\mu \e n$, and such that $S_{E} \subset \Incomp$.
  Then, for every $\theta \ge 1/\sqrt{n}$, we have
  $$
  \Pr{\inf_{x \in S_E} \|Gx\|_2 < \theta \sqrt{\e n} \text{ and } \BB_{G,M}}
  \le \left[ C (M/\sqrt{\e})^{2 \mu} \theta^{1- 2\mu} \right]^{\e n}.
  $$
\end{lemma}

This lemma implies a lower bound on the smallest singular value of $G$ restricted to the set $S_{E^-}$ similar to Lemma \ref{lem: G bounded below on E}.
\begin{corollary}  \label{cor: G bounded below general}
  Let $M \ge 1$ and $\mu \in (0,1)$.
  Then, for every $\theta \ge 1/\sqrt{n}$, we have
  $$
  \Pr{\inf_{x \in S_{E^-}} \|Gx\|_2 < \theta \sqrt{\e n}  \text{ and } S_{E^-} \subset \Incomp  \text{ and } \DD_{E^-} \cap \BB_{G,M}}
  \le \left[ C (M/\sqrt{\e})^{2 \mu} \theta^{1- 2\mu} \right]^{\e n}.
  $$
\end{corollary}

\subsection{Proof of Theorem~\ref{thm: invertibility general}}

The argument will follow the same lines as in Section~\ref{s: continuous} for continuous distributions;
here we will only indicate necessary modifications.

Without loss of generality, we may assume that
\begin{equation}  \label{eq: epsilon power of n}
 \e \ge n^{-0.4}.
\end{equation}
Indeed, \eqref{eq: t large} implies that $t^{0.4} \e^{-1.4} \ge \e^{-1} n^{-0.4}$, so the statement of Theorem \ref{thm: invertibility general} becomes vacuous whenever \eqref{eq: epsilon power of n} does not hold.

Since $\dist(B_j, H_j) \ge \dist(B'_j, H'_j)$ by \eqref{eq: BjHj},
Theorem \ref{thm: distance general} implies that
\begin{equation}         \label{eq: dist applied}
  \Pr{ \text{dist}(B_j, \text{Im}(H_j)) \le \tau \sqrt{\e n} \ \text{and} \ \BB_{A,M} }
  \le (C \tau)^{\e n}
\end{equation}
for any
\[
 \tau \ge \frac{c}{\sqrt{\e n}} + e^{-c/\sqrt{\e}}
 =: \tau_0.
\]
Consider the random variables
$$
Y_j:= \big[ \max \big( \text{dist}(B_j, H_J), \, \tau_0 \sqrt{\e n} \big) \big]^{-2} \cdot \one_{\BB_{A,M}}
$$
and argue as in Section \ref{s: dist into smi}.
We see that $Y_j$ belong to weak $L^p$ for $p = \e n/ 2$ and $\|Y_j\|_{p,\infty} \le C^2/\e n$.
By weak triangle inequality, this yields
$$
\Pr{\sum_{j=1}^n Y_j > \frac{C}{\tau^2 \e} } \le (C \tau)^{\e n}, \quad \tau > 0.
$$
Therefore,
\begin{align*}
\Pr{\sum_{j=1}^n \dist(B_j, H_j)^{-2} > \frac{1}{\tau^2 \e}\ \text{and} \ \BB_{A,M} }
&\le (C \tau)^{\e n} + \sum_{j=1}^n  \Pr{ \dist(B_j, H_j)^{-2} \neq Y_j\ \text{and} \ \BB_{A,M} } \\
&\le (C \tau)^{\e n} + \sum_{j=1}^n \Pr{ \dist(B_j, H_j) < \tau_0 \sqrt{\e n}\ \text{and} \ \BB_{A,M} }.
\end{align*}
Using again \eqref{eq: dist applied} and then \eqref{eq: epsilon power of n},
we see that this probability can be further bounded by
$$
(C \tau)^{\e n} + n (C \tau_0)^{\e n} \le (C_1 \tau)^{\e n} \quad \text{for } \tau \ge \tau_0.
$$
Defining the subspaces $E^+$, $E^-$ and the event $\DD_{E^-}$ as in Section \ref{s: E+}, we derive from this that
\begin{equation}  \label{eq: probability of D_E^- general}
 \Pr{ (\DD_{E^-})^c\ \text{and} \ \BB_{A,M} }
 \le (C_1 \tau)^{\e n} \quad \text{for } \tau \ge \tau_0.
\end{equation}

We finish the proof as in Section \ref{s: proof of invertibility continuous}. Set
\[
  \tau=\sqrt{t} \quad \text{and}  \quad
  \theta= \frac{C \sqrt{t}}{\e^{3/2}}.
\]
Then \eqref{eq: t large} ensures that $\tau \ge \tau_0$, so \eqref{eq: probability of D_E^- general} holds.
Furthermore, \eqref{eq: t large} and \eqref{eq: epsilon power of n} guarantee that $\theta \ge 1/\sqrt{n}$, hence Corollary \ref{cor: G bounded below general} applies.
Similarly to \eqref{eq: sA bar}, we use Corollary \ref{cor: G bounded below general}, Lemma \ref{lem: E- Incomp}, and \eqref{eq: probability of D_E^- general} to obtain
\begin{align*}
\Pr{s_{\bar{A}} < t \sqrt{n} \text{ and } \BB_{A,M}}
&\le \Pr{s_G < \frac{Ct}{ \tau \e} \cdot \sqrt{n} \text{ and } \BB_{A,M}} \\
&\le \Pr{s_G < \theta \cdot \sqrt{\e n} \text{ and } S_{E^-} \subset \Incomp  \text{ and } \DD_{E^-} \cap \BB_{A,M}} \\
&\quad  +\Pr{ S_{E^-} \not \subset \Incomp \text{ and } \BB_{A,M}}
     + \Pr{(\DD_{E^-})^c\ \text{and} \ \BB_{A,M} } \\
&\le \big( C  \e^{-0.05} \theta^{0.9} \big)^{\e n} + e^{-cn} + (C_1 \tau)^{\e n}  \\
&\le  \left[C  \e^{-1.4} t^{0.45} \right]^{\e n}
   + e^{-cn}  + (C t^{0.5})^{\e n}.
\end{align*}
It remains to check that the last two terms of the expression above can be absorbed into the first one. This is obvious for the third term, and follows from \eqref{eq: t large} for the second one as we assumed that $\e<c$.
This completes the proof of Theorem \ref{thm: invertibility general}. \qed

\end{document}